\documentclass[11pt, twoside]{amsart}
\usepackage{amsmath,mathtools}
\usepackage[foot]{amsaddr}
\usepackage{amsmath, amsthm, amssymb, amsfonts, enumerate}
\usepackage{todonotes}
\usepackage[colorlinks=true,linkcolor=blue,urlcolor=blue]{hyperref}
\usepackage{dsfont}
\usepackage{color}
\usepackage{geometry}
\usepackage{todonotes}
\usepackage{epstopdf}
\usepackage{bbm}
\usepackage{geometry}
\usepackage[utf8]{inputenc}
\usepackage{bm}
\usepackage{amsfonts}
\usepackage{amsfonts}
\usepackage{textcomp}
\usepackage{amssymb}
\usepackage{float}
\usepackage{tikz}
\usepackage{epsfig}
\usepackage{amsmath}
\usepackage{stmaryrd}
\usepackage{mathrsfs}
\usepackage[english]{babel}
\usepackage{a4}
\usepackage{csquotes,nicefrac}
\usepackage{enumerate}
\usepackage[multiple]{footmisc}
\usepackage[most]{tcolorbox}

\usepackage{soul}
\newcommand{\stkout}[1]{\ifmmode\text{\sout{\ensuremath{#1}}}\else\sout{#1}\fi}
\usepackage[capitalise,nameinlink]{cleveref} 
\crefname{enumi}{item}{items}
\crefname{equation}{}{}
\crefname{subsection}{Subsection}{Subsections}

\theoremstyle{definition}
\newtheorem{theorem}{Theorem}[section]

\newtheorem{remark}[theorem]{Remark}
\newtheorem{hypothesis}[theorem]{Assumption}
\newtheorem{lemma}[theorem]{Lemma}

\newtheorem{coroll}[theorem]{Corollary}
\newtheorem{definition}[theorem]{Definition}

\newtheorem{example}[theorem]{Example}

\usepackage{imakeidx}

\usepackage[backend=biber, style=numeric,maxnames=4,
  minnames=4, maxcitenames=4,
  mincitenames=4]{biblatex}
\def \E{\mathsf{E}}

\DeclareMathOperator*{\Lip}{\operatorname{Lip}}

\newcommand{\bk}[1]{\llbracket#1\rrbracket}
\DeclareMathOperator*{\esssup}{\operatorname{esssup}}
\DeclareMathOperator*{\essinf}{\operatorname{essinf}}

\definecolor{red}{rgb}{1.0,0.0,0.0}

\definecolor{blu}{rgb}{0.0,0.0,1.0}

\definecolor{gre}{rgb}{0.03,0.50,0.03}

\usepackage[T1]{fontenc}
\title{Duality for Stochastic Control  with non-Markovian Random Coefficients}

\date{\today}

\numberwithin{equation}{section}

\makeatletter
\newcommand{\authorsym}[1]{\textsuperscript{\@fnsymbol{#1}}}
\makeatother

\title{Duality for Stochastic Control  with non-Markovian Random Coefficients}

\author{Peter Bank\authorsym{1}}
\email{bank@math.tu-berlin.de}

\author{Jannis R. Dause\authorsym{1}}
\email{dause@math.tu-berlin.de}

\author{Filippo de Feo\authorsym{1}}
\email{defeo@math.tu-berlin.de}

\author{Peter K. Friz\authorsym{1}\textsuperscript{,}\authorsym{2}}
\email{friz@math.tu-berlin.de}

\thanks{\authorsym{1} Institut für Mathematik,
Technische Universität Berlin, Berlin, Germany.}

\thanks{\authorsym{2} Weierstraß-Institut,
Berlin, Germany.}

\begin{document}

\begin{abstract}We develop novel duality methods for stochastic optimal control problems under two sources of randomness and  non-Markovian random coefficients adapted to just one of them. The Hamilton--Jacobi--Bellman (HJB) equation is a second-order backward
stochastic partial differential equation.
Our  duality theory provides an alternative description of the random value function in terms of  a suitable pathwise optimal control problem parameterized by the realizations of one of the Brownian motions. This allows us to regain Markovianity using the theory of rough optimal control problems, leading to rough second order HJB equations that we solve in a suitable viscosity sense.
\end{abstract}

\maketitle

\section{Introduction}
Consider a (non-Markovian) stochastic control problem, with random value function
\begin{equation}
\label{intro:stochastic_control_prob}
    V_t(x)\coloneqq \essinf_{\pi }\mathbb E\left[\int_t^T l_s(X_s^{t,x,\pi},\pi_s) ds+g(X_T^{t,x,\pi})\bigg| \mathfrak{F}_t^W\right],
\end{equation}
and controlled doubly-stochastic dynamics given by
\begin{equation*}
    dX_s^{t, x, \pi}=b_s(X_s^{t, x, \pi},\pi_s)ds+\sigma_s (X_s^{t, x, \pi},\pi_s)dB_s+\gamma_s(X_s^{t, x, \pi},\pi_s)dW_s,\quad X_t^{t,x,\pi}=x,
\end{equation*}
where $W, B$ are two independent Brownian Motions, the coefficients depend predictably on  $W$ (but not on $B$) and $\pi$ is an admissible control. When the coefficients do not depend on $W$ either, the control problem \eqref{intro:stochastic_control_prob} reduces to the classical Markovian stochastic control setup and  can thus be treated with  Bellman's \textit{dynamic programming principle} (DPP) and the standard \textit{Hamilton-Jacobi-Bellman} (HJB) equation. The latter can be treated by the classical theory of \textit{viscosity solutions} for PDEs due to Lions-Crandall.

However, when  the coefficients depend on $W$, the problem becomes significantly more involved and was investigated by \cite{peng92,qiu2018}. A dynamic representation of \eqref{intro:stochastic_control_prob} is given as the solution (in some suitable sense) of a \textit{stochastic Hamilton--Jacobi--Bellman equation} (SHJB), i.e. a backward  stochastic partial differential equation (BSPDE) of the form
\begin{equation}
\label{intro:BSPDE}
 -dv_{t}(x)=-H_{t}\left(x,D (\mathfrak{d}_{\omega} v_t)(x),D v_t(x), D^2 v_t(x)\right)dt-\mathfrak{d}_{\omega} v_{t}(x) dW_{t}; \quad v_{T}= g .
\end{equation}
Here $Dv_t(x),D^2v_t(x)$ denote the standard derivatives in the variable $x$, while $\mathfrak{d}_{\omega} v_{t}(x)$ denotes the martingale coefficient in the semimartingale decomposition of the random field \(v_t(x)\), i.e. $
v_t( x)=v_T(x)-\int_t^T \mathfrak{d}_s v_s(x) d s-\int_t^T \mathfrak{d}_\omega v_s( x) d W_s.
$ For the study via smooth and weak solutions of the SHJB equation we refer to \cite{peng92}, while, for the characterization via viscosity solutions, we refer to  \cite{qiu2018}\footnote{These are distinct from the stochastic viscosity solutions à \cite{lions_fully_1998}}\textsuperscript{,}\footnote{See also \cite{qiu_stochastic_2023} for extensions to stochastic Hamilton-Jacobi-Bellman-Isaacs equations and \cite{qiu-yang} for extensions to infinite-dimensional state spaces.}.

 We remark that, due to possibly irregular dependence on $W,$ the control problem \eqref{intro:stochastic_control_prob} and the BSPDE \eqref{intro:BSPDE} are not covered  by the  frameworks of path-dependent PDEs (e.g. \cite{ekren2016viscosity,ekren2016viscosityII}) or path-dependent optimal control problems that can be treated via Markovian liftings and stochastic control in Hilbert spaces (e.g. \cite{bolli_defeo,defeo-fed-sw,fabbri2017stochastic} and the references therein). 

\paragraph{\textbf{Motivations.}} The complex analytic nature of SHJB \eqref{intro:BSPDE} is directly reflected in its numerical analysis, as there are no known efficient numerical schemes to solve it.  For instance, efficient numerical schemes for viscosity solutions are  open challenges even in the classical Markovian case (in this case, the SHJB equation becomes a classical PDE), when the dimension of the state space $d$ is high. The backward non-Markovian nature of such SPDE, adds further challenges.
As a close comparison, we refer to the recent paper \cite{cohen-defeo-hebner-sirignano} (and references therein) for the challenges faced by numerical methods for the easier setting of deterministic HJB equations on Hilbert spaces related to the previously mentioned path-dependent control problems via Markovian lifts: however, even these methods cannot numerically treat the SHJB equation and the related control problem \eqref{intro:stochastic_control_prob}. 

In view of this, it is crucial to investigate alternative approaches, leading to a deeper understanding of the theory and, eventually, to efficient numerical schemes. We start by observing that, in classical settings, pathwise approaches to dualize stochastic control problems are able to provide tight lower bounds on the value function, typically leading to efficient numerical schemes (see detailed literature review below). In these approaches, one writes the value function in terms of suitable pathwise optimal control problems involving penalization that have to be optimized to match the problem value.

\paragraph{\textbf{Our contributions.}} Under these motivations, in this work we provide an  approach for the stochastic control problem \eqref{intro:stochastic_control_prob} via a novel \textit{duality theory} (see \cref{sec:duality}), obtained by leveraging the analytic theory \cite{peng92,qiu2018}. In particular, the duality result provides sharp lower bounds for  \eqref{intro:stochastic_control_prob} in terms of a suitably penalized anticipative control problem. Such characterization is given both when $V$ is a smooth solution of the (SHJB) (\cref{th:duality_smooth_semimartingale}) and when this smoothness property is not assumed (\cref{th:duality_stoch_viscosity}) (inspired by the theory of viscosity solutions \cite{qiu2018}). 

Using extensions of known results from \textit{rough stochastic differential equations} (RSDEs)\cite{fhl21} the anticipative control problem  can  be identified with a rough stochastic, but Markovian control problem, which is subsequently analyzed using DPPs \cite{flz26} and path-dependent rough HJB equations, for which we establish well-posedness results under minimal regularity assumptions. The latter uses a novel approach based on the good rough paths of \cite{coutin_2007} which may be of independent interest for the analysis of path-dependent rough PDEs.  
 We also provide a conditional duality result in the standard Markovian case (\cref{sec: duality_markvian_appendix}). This   allows us to apply the arguments of the paper in the standard case, without the restrictions due to the presence of  random coefficients, and still generalizes \cite{DiehlFrizGassiat2017,bank_duality_2026}.
\paragraph{\textbf{State of the Art.}}
\textit{Duality methods} have been investigated in the context of standard optimal stopping \cite{DavisKaratzas1994, Rogers2002, bayer_primal_2025, ye-wong}, discrete-time stochastic control \cite{rogers_2007}, optimal switching \cite{ye-wong}, and standard (continuous-time) stochastic control \cite{Wets75,burstein_davis,DiehlFrizGassiat2017,henry-litterer-ren,lauriere2026, bank_duality_2026}. In particular, the robustness of the lower bounds provided by these methods  make them suitable for the development of efficient numerical schemes \cite{rogers_2007,henry-litterer-ren,lauriere2026,ye-wong}.

Pathwise stochastic control problems have been studied in \cite{burstein_davis, lions_fully_1998, buckdahn_ma_07}. In recent years pathwise control problems have also been intensively studied in the field of rough path analysis in particular for applications, e.g., in robust filtering \cite{allan_cohen_2020}, reinforcement learning \cite{chakraborty_pathwise_2024, ashkarian_pontryagin_2026}, anticipative stochastic control \cite{bank_causal_2025} and model-free portfolio theory \cite{allan_modelfree_2023}.

The theory of rough stochastic analysis arguably initiated by the first intrinsic well-posedness theory of RSDEs of \cite{fhl21} has rapidly proven to be a useful tool for applications in stochastic filtering \cite{allan_rough_2025, bugini_rough_2025}, McKean-Vlasov SDEs \cite{friz_mckean-vlasov_2025}, nonlinear PDEs \cite{bugini_nonlinear_2025}, stochastic numerics \cite{dause_controlled_2026}, option pricing \cite{bank_rough_2025}, mean-field games \cite{friz_mean-field_2026, bayraktar_mean-field_2026} and path-dependent stochastic control \cite{flz26, horst_pontryagin_2025}. In particular, we emphasize the high influence of the work \cite{flz26} on \cref{section_rough_HJB}. 
\paragraph{\textbf{Outline.}} In \cref{sec:duality}, we provide a duality theory for the optimal control problem \eqref{intro:stochastic_control_prob}. In \cref{section_rough_HJB}, we address the anticipative control problem via rough optimal control methods. 
In \cref{notations}, we collect some notations used throughout the whole paper. In \cref{sec: duality_markvian_appendix}, we  simplify the setting of the paper to doubly controlled \emph{Markovian} SDEs and we provide a duality theory in this case. In \cref{measure_theory}, we collect results from measure theory used throughout the work. In \cref{appendix_section_rsde_and_control}, we give a minimal exposition to \textit{rough stochastic analysis} and \textit{rough stochastic differential equations}(RSDEs) that we use in the paper. In \cref{kunita}, we state the It\^{o}-Kunita-Wentzell formula for smooth semimartingale fields.

\paragraph{\textbf{Acknowledgments.}} The authors are very grateful to Jinniao Qiu for helpful comments related to the present paper. JRD acknowledges funding by the Deutsche Forschungsgemeinschaft (DFG, German Research Foundation) - Project-ID 410208580 - IRTG2544 (“Stochastic Analysis in Interaction”). PB,  FdF, PKF acknowledge funding by the Deutsche Forschungsgemeinschaft (DFG, German Research Foundation) – CRC/TRR 388 ``Rough Analysis, Stochastic Dynamics and Related Fields'' – Project ID 516748464. 
\section{Duality for random coefficients}\label{sec:duality}
In this section, we introduce the optimal control problem and build our duality theory.\\
Let $m \in \mathbb{N}_{\geq 2}$ with $m=m_{B}+ m_{W}$ for $ m_{B}, m_{W}\in \mathbb{N}$.  Let $\Omega = C([0,T]; \mathbb{R}^{m})$ and $(B, W)$ the $m_{B}+m_{W}$-dimensional canonical process on the Wiener space $(\Omega, \mathfrak{F}, (\mathfrak{F}_{t})_{t \in [0,T]}, \mathbb{P})$, where $\mathfrak{F}=\mathfrak{B}(\Omega)$ and  $(\mathfrak{F}_{t})_{t \in [0,T]}$ is the augmented Brownian filtration of $(B, W)$. Then $B$ and $W$ are respectively $m_{B}, m_{W}$-dim. independent Brownian Motions on the above probability space with their respective augmented  filtrations given by $(\mathfrak{F}^{B}_{t})_{t\in [0,T]}$, $(\mathfrak{F}^{W}_{t})_{t \in [0,T]}$.
We consider a controlled SDE 
\begin{equation}
\label{eq:controlled_doubly_sde}
    dX_s^{t, x, \pi}=b_s(X_s^{t, x, \pi},\pi_s)ds+\sigma_s (X_s^{t, x, \pi},\pi_s)dB_s+\gamma_s(X_s^{t, x, \pi},\pi_s)dW_s , \quad X_t^{t, x, \pi}=x\in \mathbb R^d, 
\end{equation}
where $\pi\in \mathcal U=\{\pi \colon [0, T]\times \Omega \to U \textit{ is }(\mathfrak{F}_s) \textit{-predictable} \}$ are admissible controls, $U \subset \mathbb R^h$ compact, and $b: [0, T] \times \Omega \times \mathbb{R}^d \times U \rightarrow \mathbb{R}^d$, $\sigma: [0,T]\times \Omega \times \mathbb{R}^d \times U \rightarrow \mathbb{R}^{d \times m_{B}}$, $\gamma: [0, T] \times \Omega \times \mathbb{R}^d \times U \rightarrow \mathbb{R}^{d \times m_{W}}$ will be predictable with respect to $(\mathfrak{F}^W_s)$, in a sense specified  below. Denote  $\Sigma=[\sigma,\gamma]:[0, T] \times \Omega \times \mathbb{R}^d \times U \rightarrow \mathbb{R}^{d \times m}$, with $m=m_{W}+m_{B}$.

Consider the following optimal control problem with random value function
\begin{align*}
V_t(x)(\omega)&:=\essinf_{\pi \in \mathcal U}\mathbb E\left[\int_t^T l_s(X_s^{t,x,\pi},\pi_s) ds+g(X_T^{t,x,\pi})\bigg| \mathfrak{F}_t\right](\omega)\\
&=\essinf_{\pi \in \mathcal U}\mathbb E\left[\int_t^T l_s(X_s^{t,x,\pi},\pi_s) ds+g(X_T^{t,x,\pi})\bigg| \mathfrak{F}_t^W\right](\omega),
\end{align*}
where $l:  [0, T] \times \Omega \times \mathbb{R}^d \times U \rightarrow \mathbb{R}$, $g: \Omega \times \mathbb{R}^d \rightarrow \mathbb{R}$ are $(\mathfrak{F}^{W}_t)$-predictable (resp. $\mathfrak{F}^W_T$-measurable).
Throughout this section we will always work under the following assumption:
\begin{hypothesis}\label{hp:regularity_coeff}Let 
$g \in L^{\infty}\left(\Omega ; H^{1, \infty}\right)$. For the coefficients $z=b^i, \Sigma^{i j},l ,$ for $1 \leq i \leq d, 1 \leq j \leq m$, we assume that
\begin{enumerate}
    \item $z$ is $\mathfrak{P}^W \otimes \mathfrak{B}\left(\mathbb{R}^d\right) \otimes \mathfrak{B}(U)$-measurable;
    \item for $dt \otimes \mathbb{P}$-almost all $(t, \omega),$ $ z_t( x, \pi)$ is uniformly continuous on $\mathbb{R}^d \times U$;
    \item there exists $L>0$ such that
    $\|g\|_{L^{\infty}\left(\Omega; H^{1, \infty}\right)}+\sup _{\pi \in \mathcal{U}}\|z_{(\cdot)}( \cdot, \pi)\|_{\mathcal{L}^{\infty}\left(H^{1, \infty}\right)} \leq L$.
\end{enumerate}
\end{hypothesis}
Under \cref{hp:regularity_coeff}, \eqref{eq:controlled_doubly_sde} and $V$ are well-posed  \cite[Section 3]{qiu2018}.

\subsection{The smooth setting}
Throughout this section we will work under the following assumption (recall  \cref{def:smooth-semimartingales}):
\begin{hypothesis}
    \label{ass:peng}
    $V \in \mathcal{C}^{2}_{\mathfrak{F}}$ is a continuous semimartingale field of the form
\begin{align}
V_t(x)=g(x)-\int_t^T \mathfrak{d}_{s} V_s(x) d s-\int_t^T \mathfrak{d}_{\omega} V_s(x) d W_s,\quad t\in [0,T],x\in \mathbb R^d
\end{align}
and the processes $\mathfrak{d}_{t} V,\mathfrak{d}_{\omega} V$ are  continuous in $(t,x)$, $\mathbb P$-a.s.
\end{hypothesis}
The random value function leads us 
to consider the following stochastic Hamilton-Jacobi-Bellman (SHJB) equation: 
$$
-d v_t(x)=H_t\left(x,D \mathfrak{d}_{\omega} v_t(x),D v_t(x), D^2 v_t(x)\right) d t-\mathfrak{d}_{\omega} v_t(x) d W_t, \quad v_T(x)=g(x),
$$
i.e. 
$$v_t(x)=g(x)+\int_t^T H_s\left(x,D \mathfrak{d}_{\omega} v_s(x),D v_s(x), D^2 v_s(x)\right) d s-\int_t^T \mathfrak{d}_{\omega} v_s(x) d W_s, $$
where
$$
\begin{aligned}
&H_t(x, q,p,A ):=\inf _{u \in U}H^{\text{cv}}_t(x, q,p,A ,u),\\
&H^{\text{cv}}_t(x,q,p,A ,u):=\gamma_t(x,u)\cdot q+b_t(x,u)\cdot p+l_t(x,u)+\operatorname{tr}[a_t(x,u)A],
\end{aligned}
$$
with 
$
a_t(x,u)=\frac{1}{2}\Sigma_t (\Sigma_t)^{\top
}(x, u)=\frac{1}{2}\left[\sigma_t (\sigma_t)^{\top}(x, u)+\gamma_t (\gamma_t)^{\top}(x, u)\right]$. 

Proceeding as in \cite[Section 3]{peng92},
    under \cref{ass:peng}, we have
\begin{align}\label{eq.consequence_shjb}
\mathfrak{d}_{s} V_{s}(x)=-H_s\left(x,D (\mathfrak{d}_{\omega} V_s)(x),D V_s(x), D^2 V_s(x)\right).
\end{align}
Comparing  \cref{ass:peng} with (SHJB), it follows that the couple  $\left(V_t(x), \mathfrak{d}_{\omega} V_t(x)\right)$ satisfies (SHJB).

For the subsequent results, we introduce the following notations. The It\^{o}-Kunita-Wentzell formula allows us to associate with any $h \in \mathcal C^2_\mathfrak{F}$, the $(\mathfrak F_t)$-martingale $M^{t,x,\pi,h}$ given by
\begin{align*}
M_{t,\tau}^{t,x,\pi,h} &:=h_\tau(X_\tau^{t,x,\pi})-h_t(x)-\int_t^\tau\big[h_s^0(X_s^{t,x,\pi})+b_s(X_s^{t,x,\pi},\pi_s)  D h_s(X_s^{t,x,\pi})\\
&\quad + \operatorname{tr} \left( a_s(X_s^{t,x,\pi},\pi_s) D^2 h_s(X_s^{t,x,\pi}) \right) + \gamma_s(X_s^{t,x,\pi},\pi_s)  Dh_s^1(X_s^{t,x,\pi},\pi_s)\big] ds\\
&=\int_t^\tau Dh_s(X_s^{t,x,\pi}) \sigma_s(X_s^{t,x,\pi},\pi_s) dB_s+ \int_t^\tau\left[ h_s^1(X_s^{t,x,\pi}) + Dh_s(X_s^{t,x,\pi})\gamma_s(X_s^{t,x,\pi},\pi_s) \right]dW_s,
\end{align*}
Moreover, we set
\begin{align}
& V^{1,h}_t(x)\coloneqq \mathbb E\Big[\essinf_{\pi \in {\mathcal U}} \mathbb E\left[ \int_t^T l_s(X_s^{t,x,\pi},\pi_s) ds+g(X_T^{t,x,\pi}) - M_{t,T}^{t,x,\pi,h} \bigg| \mathfrak{F}_T^W\right] \mathfrak{F}_t^W \Bigg]\nonumber\\
& \quad\quad\quad\quad = h_t(x) + \mathbb{E}\left[\essinf_{\pi \in {\mathcal U}}\mathbb E\left[ \int_t^T\big[ \mathfrak{d}_{s} h_{s}(X^{t,x, \pi}_{s})\right. \right. \nonumber \\
& \quad \quad\quad\quad+\left. \left. H^{\text{cv}}_s(X^{t,x, \pi}_{s},D (\mathfrak{d}_{\omega}h_s)(X^{t,x, \pi}_{s}),D h_s(X^{t,x, \pi}_{s}) ,D^2 h_s(X^{t,x,\pi}_{s}),\pi_s)\big]ds+g(X_T^{t,x,\pi})-h_T(X_T^{t,x,\pi})    \bigg| \mathfrak{F}_T^W\right] \bigg| \mathfrak{F}_t^W \right], \label{eq:duality_th}\\
& V^{2,h}_t(x)\coloneqq h_t(x) +\mathbb E\bigg[ \int_t^T  \essinf_{y}\left[\mathfrak{d}_{s}h_{s}(y)+H_s\left(y,D( \mathfrak{d}_{\omega}h_{s})(y),D h_s(y), D^2 h_s(y)\right) \right]ds \nonumber \\
&\quad \quad\quad\quad\quad \quad \quad\quad\quad\quad\quad \quad\quad\quad \quad\quad +\essinf_{y}\left[g(y)-h_T(y) \bigg] \bigg| \mathfrak{F}_t^W  \right],\nonumber 
\end{align}
for all $h \in \mathcal C^2_\mathfrak{F}$.

We are now ready to state our first duality result.
\begin{theorem}[Duality, smooth setting]\label{th:duality_smooth_semimartingale}    Let Assumptions \ref{hp:regularity_coeff} and \ref{ass:peng} hold. Then almost surely
\begin{align}
\label{eq:duality_smooth_case}
V_t(x)&=\max_{h \in \mathcal{H}}  V^{1,h}_t(x)=\max_{h \in \mathcal{H}}  V^{2,h}_t(x),
\end{align} where $\mathcal{H}:= \{ h \in \mathcal{C}^{2}_{\mathfrak{F}} : h_{T}= g \}$.
\end{theorem}
The strength of this theorem is that it does not merely give an abstract reformulation of the value function, but produces sharp dual bounds for it: for every sufficiently smooth test semimartingale $h$, the corresponding dual functional gives a pathwise optimization problem in which the $W$-noise is frozen and the martingale correction compensates for the enlarged, potentially anticipative information. This yields a rigorous bound on $V_t(x)$, while the choice $h=V$ makes the bound exact, so there is no duality gap. Of course, the choice $h=V$ is only possible when the value function is smooth in the sense that $V \in \mathcal H$. However,  in the next subsection, we will also provide a duality result  under non-smoothness of $V$.
In any case, when the true value function is unavailable, one can approximate it by tractable test functions $h$ and obtain bounds that improve as the approximation improves. In this sense, the theorem transforms the difficult (SHJB) problem into a family of pathwise control problems whose optimal dual element recovers the true value. We will analyze the pathwise control problem more closely in \cref{section_rough_HJB}. 
\begin{proof}
For all $h \in \mathcal{H}$, and using the It\^{o}-Kunita-Wentzell formula (\cref{lemma:It\^{o}-kunita}), we have almost surely
\begin{align*}
V_t(x)&=\essinf_{\pi \in {\mathcal U}} \mathbb E\left[\int_t^T l_s(X_s^{t,x,\pi},\pi_s) ds+h_T(X_T^{t,x,\pi})  \bigg| \mathfrak{F}_t^W\right]\\
&=h_t(x)+ \essinf_{\pi \in {\mathcal U}} \mathbb E\bigg[ \int_t^T \big[\mathfrak{d}_{s}h_{s}(X_s^{t,x,\pi})\\
&\quad \quad\quad\quad+  H^{\text{cv}}_s(X_s^{t,x,\pi},D (\mathfrak{d}_{\omega
}h)_s(X_s^{t,x,\pi}),D h_s(X_s^{t,x,\pi}) ,D^{2} h_s(X_s^{t,x,\pi}),\pi_s ) \big]  ds  \bigg| \mathfrak{F}_t^W  \bigg]\\
&=h_t(x)+ \essinf_{\pi \in {\mathcal U}}\mathbb E\bigg[  \mathbb E\bigg[ \int_t^T \big[ \mathfrak{d}_{s}h_{s}(X_s^{t,x,\pi}) \\
&\quad \quad\quad\quad +H^{\text{cv}}_s(X_s^{t,x,\pi},D (\mathfrak{d}_{\omega}h)_s(X_s^{t,x,\pi}),D h_s(X_s^{t,x,\pi}) ,D^2 h_s(X_s^{t,x,\pi}),\pi_s ) \big] ds \bigg| \mathfrak{F}_T^W\bigg] \bigg| \mathfrak{F}_t^W  \bigg]\\
&\geq  V^{1,h}_t(x)\geq  V^{2,h}_t(x), 
\end{align*}
where in the first inequality of the last line we moved the essinf inside the first conditional expectation, see \cref{cor:conditional_essinf} for a rigorous justification. Taking $h=V\in \mathcal C^2_\mathfrak{F}$ and using \eqref{eq.consequence_shjb} (which is a consequence of the (SHJB) equation), we have
$ V_t(x)= V^{2,V}_t(x)$ almost surely. 
The claim follows.
\end{proof}
\subsection{The non-smooth setting}
\label{section_duality_for_viscosity_solutions} In this subsection, inspired by \cite{qiu2018}, we relax the smoothness assumption on $V,$ i.e. \cref{ass:peng}.

 As in \cite{qiu2018}, we define the space of smooth pointwise sub-, super-solutions, respectively, by
 $$
\begin{aligned}
& \underline{\mathcal{S}}:=\left\{\phi \in \mathcal{C}_{\mathfrak{F}}^2: \phi^{+} \in \mathcal{S}^{\infty}(C_0(\mathbb{R}^d)), \phi_T(x) \leq g(x),\right. \\
& \quad \quad \quad \quad\quad \quad\quad  -\mathfrak{d}_t \phi_t(x)-H_t\left( x, D \mathfrak{d}_\omega \phi_t(x), D \phi_t(x), D^2 \phi_t(x)\right)\leq 0 \left. \quad (t, x) \in[0, T) \times \mathbb{R}^d\right\}, \\
&\overline{\mathcal{S}}:=\left\{\phi \in \mathcal{C}_{\mathfrak{F}}^2: \phi^{-} \in \mathcal{S}^{\infty}(C_0(\mathbb{R}^d)), \phi(T, x) \geq G(x),\right.\\
&\quad \quad \quad \quad\quad \quad\quad -\mathfrak{d}_t \phi_t(x)-H_t\left( x, D \mathfrak{d}_\omega \phi_t(x), D \phi_t( x), D^2 \phi_t( x)\right) \geq 0 \left. \quad (t, x) \in[0, T) \times \mathbb{R}^d\right\}.
\end{aligned}
$$
\begin{theorem}[Duality, non-smooth setting]\label{th:duality_stoch_viscosity}Let \cref{hp:regularity_coeff} hold. Let $t\in [0,T],x\in \mathbb R^d$ such that almost surely
\begin{equation}\label{eq:essup_phi_getV}
    \esssup_{\phi \in \underline{\mathcal{S}}} \phi_{t}(x) \geq V_{t}(x).
\end{equation}
Then, we have almost surely
\begin{align}
\label{eq:duality_viscosity}
V_t(x)&=\esssup_{h \in \mathcal C^2_\mathfrak{F}}  V^{1,h}_t(x)=\esssup_{h \in \mathcal C^2_\mathfrak{F}}  V^{2,h}_t(x).
\end{align}
\end{theorem}
\begin{remark}
Under \cref{hp:regularity_coeff}, additional regularity assumptions on the coefficients and using viscosity solutions, the proof of \cite[Theorem 5.6]{qiu2018}  shows that\footnote{see Equation (5.5) there, together with the successive lines, and the end of the proof} $$\esssup_{\phi \in \underline{\mathcal{S}}} \phi_{t}(x) =V_{t}(x)=\essinf_{\phi \in \overline{\mathcal{S}}} \phi_{t}(x),\quad \text{for all } t\in [0,T],x\in \mathbb R^d$$ so that \eqref{eq:essup_phi_getV} is satisfied for all $t,x$.
\end{remark}
\begin{proof}[Proof of \cref{th:duality_stoch_viscosity}]
For all $h \in \mathcal C^2_\mathfrak{F}$ (recall \cref{def:smooth-semimartingales}), by applying the It\^{o}-Kunita-Wentzell formula, we have almost surely
\begin{align*}
V_t(x)&=\essinf_{\pi \in {\mathcal U}} \mathbb E\left[\int_t^T l_s(X_s^{t,x,\pi},\pi_s) ds+h_T(X_T^{t,x,\pi})+g(X_T^{t,x,\pi})-h_T(X_T^{t,x,\pi})  \bigg| \mathfrak{F}_t^W\right]\\
&=h_t(x)+ \essinf_{\pi \in {\mathcal U}} \mathbb E\bigg[ \int_t^T \big[\mathfrak{d}_{s} h_{s}(X_s^{t,x,\pi}) +  H^{\text{cv}}_s(X_s^{t,x,\pi},D (\mathfrak{d}_{\omega} h)_s(X_s^{t,x,\pi}),D h_s(X_s^{t,x,\pi}) ,D^{2} h_s(X_s^{t,x,\pi}),\pi_s )\big] ds  \\
&\quad +g(X_T^{t,x,\pi})-h_T(X_T^{t,x,\pi})  \bigg| \mathfrak{F}_t^W  \bigg]\\
&=h_t(x)+ \essinf_{\pi \in {\mathcal U}}\mathbb E\bigg[  \mathbb E\bigg[ \int_t^T \big[ \mathfrak{d}_{s}h_{s}(X_s^{t,x,\pi}) +H^{\text{cv}}_s(X_s^{t,x,\pi},D (\mathfrak{d}_{\omega}h)_s(X_s^{t,x,\pi}),D h_s(X_s^{t,x,\pi}) ,D^2 h_s(X_s^{t,x,\pi}),\pi_s )\big]ds \\
&\quad +g(X_T^{t,x,\pi})-h_T(X_T^{t,x,\pi})  \bigg| \mathfrak{F}_T^W\bigg] \bigg| \mathfrak{F}_t^W  \bigg]\\
&\geq \esssup_{h \in \mathcal C^2_\mathfrak{F}}  V^{1,h}_t(x)\geq \esssup_{h \in \mathcal C^2_\mathfrak{F}}  V^{2,h}_t(x)\geq \esssup_{h \in \underline{\mathcal S}} h_t(x)\geq V_t(x),
\end{align*}
where in the last line we have used the fact that $\mathcal C^2_\mathfrak{F} \supset \underline{\mathcal S}$, \eqref{eq:essup_phi_getV} (and  \cref{cor:conditional_essinf}).
The claim follows.
\end{proof}
\section{Anticipative stochastic control with random coefficients via rough paths}
\label{section_rough_HJB}
In this section, we show how to address the anticipative control problem associated with $V^{1,h}$ for a fixed $h\in \mathcal H$ (resp. $h\in \mathcal{C}^{2}_{\mathfrak{F}}$ in the context of \cref{th:duality_stoch_viscosity}) via rough optimal control methods.
\subsection{Anticipative stochastic control and RSDEs}
\label{subsection_the_anticipative_control_problem_and_randomization_of_RSDEs}
    Motivated by \eqref{eq:duality_smooth_case} and \eqref{eq:duality_viscosity}, consider the partially anticipative problem:   
    \begin{equation}
    \label{eq:path_dep_control_problem}
  \hat V_t(x):= \essinf_{\pi \in {\mathcal U}}\mathbb E\left[ \int_t^T \hat l_s(X_s^{t,x,\pi},\pi_s) ds +\hat g(X_T^{t,x,\pi})  \bigg| \mathfrak{F}_T^W\right], 
  \end{equation}
    where $\hat l, \hat g$ are $(\mathfrak{F}_{t}^{W})$-predictable resp. $\mathfrak{F}^{W}_{T}$-measurable. For example in the context of \cref{th:duality_stoch_viscosity}, \eqref{eq:duality_th} take  $\hat l_s(x,\pi):=\mathfrak{d}_{s}h_{s}(x) +H_s^{\text{cv}}(x,D (\mathfrak{d}_{\omega}h_s)(x),D h_s(x) ,D^2 h_s(x),\pi), \hat g(x):=g(x)-h_T(x)$ to see that $\hat V=V^{1,h}$. Such pathwise stochastic control problems have previously been considered e.g. in \cite{burstein_davis, lions_fully_1998,  buckdahn_ma_07} and more recently \cite{flz26}. Following the latter, \eqref{eq:path_dep_control_problem} can be treated with techniques coming from the recently developed theory of controlled rough SDEs see \cite{fhl21, flz26, flz25} (and therefore avoiding subtle measurability issues present in the former works). Note however, that since our setting involves random terminal cost $\hat{g}$ and running-cost $\hat{l}$ as well as random coefficients in the dynamics we still need to verify that randomization procedures as in \cite{flz26} and \cite{flz25} are admissible here. 
    \newline
    Throughout this section we consider a filtered probability space 
    \begin{equation*}
        (\Omega , \mathfrak{F}, (\mathfrak{F}_{t})_{t \in [0,T]}, \mathbb{P})= (\Omega', \mathfrak{F}', (\mathfrak{F}'_{t})_{t \in [0,T]}, \mathbb{P}')\otimes (\Omega'', \mathfrak{F}'', (\mathfrak{F}_{t}'')_{t \in [0,T]}, \mathbb{P}''),
    \end{equation*}
    where the former is the augmented canonical space generated by the $m_{B}$-dim. canonical process denoted by $B$ and the latter is the augmented canonical space generated by $m_{W}$-dim. canonical process $W$. See \cref{appendix:subsection:product_prob_spaces} for details on this construction. In the following let $b : [0,T] \times \Omega \times \mathbb{R}^{d} \times U \to \mathbb{R}^{d}$; $\sigma: [0,T] \times \Omega \times \mathbb{R}^{d} \times U \to \mathbb{R}^{d \times m_{B}}$ and $\gamma: [0,T] \times \Omega \times \mathbb{R}^{d} \times U \to \mathbb{R}^{d \times m_{W}}$. Further we need to impose the following condition on the measurability structure:
    \begin{hypothesis}
    \label{progressive_assumption}
        For the coefficients $\Theta \in \{b, \gamma, \sigma , \hat{l}, \hat g \}$ and a suitable choice of target space 
        $V\in \{ \mathbb{R}^{d}, \mathbb{R}^{d\times m_{W}}, \mathbb{R}^{d\times m_{B}}, \mathbb{R} \}$ respectively, $\Theta: [0,T] \times \Omega \to C_{b}(\mathbb{R}^{d} \times U; V)$
        is strongly \\
        $\mathfrak{P}^{W}/\mathfrak{B}(C_{b}(\mathbb{R}^{d} \times U; V))$-measurable, where $\mathfrak{P}^{W}$ denotes the predictable $\sigma$-algebra w.r.t. the filtration $(\mathfrak{F}^{W}_{t})$ on $[0,T]\times \Omega$. 
    \end{hypothesis}
    \begin{lemma}
    \label{lemma: deterministic_version}Suppose \cref{progressive_assumption} holds.
    Let $\Theta \in \{b, \sigma, \gamma, D_{x} \gamma, \hat{l}\}$ such that $\Theta: [0,T]\times \Omega \times \mathbb{R}^{d}\to V$ for a suitable target space $V\in \{ \mathbb{R}^{d}, \mathbb{R}^{d\times m_{W}}, \mathcal{L}(\mathbb{R}^{d}; \mathbb{R}^{d\times m_{W}}), \mathbb{R}^{d\times m_{B}}, \mathbb{R} \}$. Then there exist strongly Borel-measurable $\theta^{\text{det}}: [0,T]\times C([0,T]; \mathbb{R}^{m_{W}}) \to C_{b}(\mathbb{R}^{d} \times U; V)$ and $\hat{g}^{\text{det}}: \mathbb{R}^{d} \times C([0,T]; \mathbb{R}^{m_{W}})\to \mathbb{R}$ such that 
    \begin{equation*}
    \Theta_{t}(\cdot, \omega) = \Theta^{\text{det}}_{t}(\cdot, W(\omega))=\Theta^{\text{det}}_{t}(\cdot, W_{\cdot \wedge t}(\omega)); \quad \hat{g}(\cdot, \omega)= \hat{g}^{\text{det}}(\cdot, W(\omega))
    \end{equation*}
    as elements in $ C_{b}(\mathbb{R}^{d} \times U; V)$ for $\mathbb{P}$-a.e. $\omega \in \Omega$, for every $t \in [0,T]$. In particular this implies that for any $(x,u) \in \mathbb{R}^{d} \times U$ it holds
    \begin{equation*}
        \Theta(t, x, u, \omega)= \Theta^{\text{det}}_{t}(x, u,  W_{\cdot \wedge t}(\omega)); \; \; \; \; \hat{g}(x, \omega)= \hat{g}^{\text{det}}(x, W_{\cdot \wedge T}(\omega))
    \end{equation*}
    for any $t \in [0,T]$, for $\mathbb{P}$-a.e. $\omega \in \Omega$
    \end{lemma}
    \begin{proof}
        Note that since $\mathfrak{F}^{W}_{t}= (\{ \emptyset , \Omega' \} \otimes \mathfrak{F}''_{t})\vee \mathcal{N}$ where $\mathcal{N}$ denotes the collection of $\mathbb{P}=\mathbb{P}'\otimes \mathbb{P}''$-nullsets, we see that by \cref{progressive_assumption}, for any $\Theta$ as above there is a $(\mathfrak{F}''_{t})$-predictable $\bar{\Theta}: [0,T]\times \Omega''\to C(\mathbb{R}^{d}\times U; V)$ such that $\Theta=\bar{\Theta}$ up to $\mathbb{P}$-indistinguishability. The existence of $\Theta^{\text{det}}$ as described above is then a direct consequence of applying \cref{appendix:lemma:optional_nonantic} to $\bar{\Theta}$.
    \end{proof}
    Notably \cref{lemma: deterministic_version} allows us to work with coefficients as deterministic, non-anticipative functionals on path space instead of random fields. From now on we will not distinguish between  $\Theta$ and $\Theta^{\text{det}}$ but w.l.o.g. only work with the latter one. Note that the solution of the controlled doubly-SDE 
    \begin{equation}
    \label{eq:dcSDE}
    \begin{aligned}
    dX_{s}^{t, x, \pi}=&b_s^{\text{det}}(X_s^{t, x, \pi},\pi_s, W_{\cdot \wedge s})ds+\sigma_s^{\text{det}} (X_s^{t, x, \pi}, \pi_s, W_{\cdot \wedge s})dB_s\\
    &+\gamma_s^{\text{det}}(X_s^{t, x,\pi},\pi_s, W_{\cdot \wedge s})dW_s , \quad X_t^{t, x, \pi}=x\in \mathbb R^d
    \end{aligned}
    \end{equation} is indistinguishable from the solution of \eqref{eq:controlled_doubly_sde} and thus will be denoted by the same letter going forward. Naturally, for suitably regular $\gamma$, we can also rewrite \eqref{eq:dcSDE} in Stratonovich-form 
    \begin{equation*}
    \begin{aligned}
    dX_{s}^{t, x, \pi}=&\tilde{b}_s(X_s^{t, x, \pi}, \pi_s, W_{\cdot \wedge s})ds+\sigma_s (X_s^{t, x, \pi}, \pi_s, W_{\cdot \wedge s})dB_s+\gamma_s(X_s^{t, x,\pi},\pi_s, W_{\cdot \wedge s})\circ dW_s , \quad X_t^{t, x, \pi}=x\in \mathbb R^d. 
    \end{aligned}
    \end{equation*}
Here we consider the corrected drift 
\begin{equation*}
    \tilde{b}_{s}(x,u, z)\coloneqq b_{s}(x, u,z) -\frac{1}{2} \left( \sum_{k=1}^{m_{W}}(D_{x}\gamma_{s}(x, u, z)\gamma_{s}(x, u, z) +  \gamma'_{s}(x,u,z))(e_{k}\otimes e_{k})\right)
\end{equation*}
with $\gamma'$ being the Gubinelli-derivative of $\gamma$ w.r.t $W$ (see \cref{assumption_RSDE_existence_uniqueness} for details). We further need the following assumptions on the coefficients which follow from \cref{thm.fixpoint} to ensure existence and uniqueness of solutions to the controlled RSDEs: 
\begin{equation}
    \label{eq:controlled_RSDE}
\begin{aligned}
    dX_{s}^{t, x, \pi; \mathbf{Z}}= &b_{s} (X_{s}^{t, x, \pi; \mathbf{Z}}, \pi_{s}, Z_{\cdot \wedge s} ) ds + \sigma_{s} (X_{s}^{t, x, \pi; \mathbf{Z}}, \pi_{s},  Z_{\cdot \wedge s}) dB_{s}+ \gamma_{s}(X_{s}^{t, x, \pi; \mathbf{Z}},  Z_{\cdot \wedge s}) d\mathbf{Z}_{s}; \quad X^{t, x, \pi; \mathbf{Z}}_{t}=x \\
    d\tilde{X}_{s}^{t, x, \pi; \mathbf{Z}}= & \tilde{b}_{s} (\tilde{X}_{s}^{t, x, \pi; \mathbf{Z}}, \pi_{s}, Z_{\cdot \wedge s} ) ds + \sigma_{s} (\tilde{X}_{s}^{t, x, \pi; \mathbf{Z}}, \pi_{s},  Z_{\cdot \wedge s}) dB_{s}+ \gamma_{s}(\tilde{X}_{s}^{t, x, \pi; \mathbf{Z}},  Z_{\cdot \wedge s}) d\mathbf{Z}_{s}; \quad \tilde{X}^{t, x, \pi; \mathbf{Z}}_{t}=x
\end{aligned}
\end{equation}
where $\mathbf{Z} =(Z, \mathbb{Z}) \in \mathscr{C}^{\alpha}([0,T]; \mathbb{R}^{m_{W}})$ and $\pi$ denotes some $U$-valued control to be specified in \cref{def:controls}. 
    \begin{hypothesis}
    \label{assumption_RSDE_existence_uniqueness}
        Suppose that the following hold:
        \begin{enumerate}[(i)]
        \item $\gamma$ is independent of any controls $\pi$. 
        \item For any $Z \in C([0,T]; \mathbb{R}^{m_{W}})$ and $\pi \in \mathcal{U}$ it holds for $\Theta \in \{b(\cdot, \pi), \sigma(\cdot, \pi), \hat{l}(\cdot, 
        \pi) \}$, that
        \begin{equation*}
            \Theta^{Z, \pi} : [0,T] \times \Omega \times \mathbb{R}^{d} \to V; \, \; \; \; (t, \omega, x) \mapsto \Theta_{t}(t, x, \pi_{t}(\omega), Z_{\cdot \wedge t})
        \end{equation*}
        is $(\mathfrak{F}_{t})$-predictable and random bounded Lipschitz as in \cref{def:random_bounded_Lipschitz}.
        \item Let $\beta > \nicefrac{1}{\alpha}$. There is $\gamma'$ such that 
        $
           (\gamma, \gamma'): [0,T] \times \mathbb{R}^{d} \times C^{m_{W}} \to \mathcal{L}(\mathbb{R}^{m_{W}}; \mathbb{R}^{d})\times \mathcal{L}(\mathbb{R}^{m_{W}}\otimes \mathbb{R}^{m_{W}}; \mathbb{R}^{d})$,  
        such that for any $\mathbf{Z}=(Z, \mathbb{Z})\in \mathscr{C}^{\alpha}([0,T]; \mathbb{R}^{m_{W}})$, $f\in \{\gamma, \gamma'\}$ $f_{t}(\cdot, Z)= f_{t}(\cdot, Z_{\cdot \wedge t})$. Further assume $(\gamma(\cdot, Z), \gamma'(\cdot, Z)) \in \mathscr{D}_{Z}^{2\alpha} C_{b}^{\beta}$ and $(D_{x}\gamma(\cdot, Z), D_{x}\gamma'(\cdot, Z)) \in \mathscr{D}_{Z}^{2\alpha} C_{b}^{\beta-1}$.  
        \item For any $Z \in C([0,T]; \mathbb{R}^{m_{W}})$, $\hat{g}(\cdot, Z)\in \operatorname{BUC}(\mathbb{R}^{d})$. 
        \end{enumerate}
    \end{hypothesis}
    \begin{remark}
    \label{remark:controlled_diffusion}
    Notably $(\gamma, \gamma')$ do not depend on the control $\pi$, as the problem degenerates otherwise due to the unbounded variation of $Z$, see \cite[Remark 13]{DiehlFrizGassiat2017}. In \cite{allan_cohen_2020, IannucciCrisanCass2025} the authors provide criteria to consider such controlled (in sense of $\pi$ dependence) diffusion coefficients in the rough term, however at the expense of restricting the class of admissible controls and of adding regularization terms to the cost functional. 
    \end{remark}
    As in \cite{flz26}, we want to distinguish between controls with multiple degrees of dependence on $\mathbf{Z}$. Therefore we introduce the following sets of controls: 
    \begin{definition}[\cite{flz26} Section 7]
    \label{def:controls}
        Let $\mathfrak{P}'$ be the predictable $\sigma$-algebra on $[0,T] \times \Omega'$ w.r.t. $(\mathfrak{F}'_{t})_{t \in [0,T]}$. Consider  the following sets of $U$-valued controls 
        $\pi : [0,T] \times \Omega' \times \mathscr{C}^{0, \alpha, 1}([0,T]; \mathbb{R}^{m_{W}})\to U$: 
        \begin{equation*}
            \begin{aligned}
                \mathcal{A}^{1}&\coloneqq \{\pi \; \; \text{is} \; \; \mathfrak{P}'\otimes \{\emptyset, \mathscr{C}^{0, \alpha, 1}_{T} \}/\mathfrak{B}_{U} \; -\text{measurable} \}\\
                \mathcal{A}^{2}&\coloneqq \{\pi \; \; \text{is} \; \; \mathfrak{P}'\otimes \mathfrak{C}^{\alpha}_{T}/\mathfrak{B}_{U} \; -\text{measurable} \}\\
                \mathcal{A}^{c}&\coloneqq \{\pi \; \; \text{is} \; \; \mathfrak{P}'\otimes \mathfrak{C}^{\alpha}_{T}/\mathfrak{B}_{U} \; -\text{measurable}; \; \; \pi_{t}(\omega'; \mathbf{Z})= \pi_{t}(\omega'; \mathbf{Z}_{\cdot \wedge t}) \; \text{for} \; \mathbb{P}'-\text{a.e.} \; \omega' \in \Omega' \}\\
                \mathcal{A}&\coloneqq \{ \pi : [0,T]\times \Omega' \to U \; \text{is} \; \mathfrak{P}'/\mathfrak{B}_{U} \; -\text{measurable}\}  
            \end{aligned}
        \end{equation*}
        Obviously it holds $\mathcal{A}^{1} \subset \mathcal{A}^{c} \subset \mathcal{A}^{2}$. For any such $
        \pi \in \mathcal{A}^{i}$ for $i\in \{1, 2, c \}$ we define $\bar{\pi}: [0,T]\times \Omega \to U; (t, \omega)\mapsto \pi(t, \omega', \mathbf{W}^{\text{It\^{o}}}(\omega''))$, where $\mathbf{W}^{\text{It\^{o}}}$ is the It\^{o}-lift of $W$. 
    \end{definition}
    We are now equipped to discuss existence, uniqueness and measurability of solutions to \eqref{eq:controlled_RSDE}. Recall that $\mathfrak{C}^{\alpha}_{T}$ denotes the Borel-$\sigma$-algebra of the Polish space $\mathscr{C}^{0, \alpha, 1}([0,T]; \mathbb{R}^{m_{W}})$. 
    \begin{theorem}
    \label{rsde_existence_uniqueness_progressiveversion}
        Suppose \cref{assumption_RSDE_existence_uniqueness} holds. Then for any $q \in [2, \infty)$, $(t, x) \in [0, T]\times \mathbb{R}^{d}$, any $\pi \in \mathcal{A}^{i}$, $i\in \{1,2, c\}$ and any $\mathbf{Z}=(Z, \mathbb{Z}) \in \mathscr{C}^{0, \alpha, 1}([0,T]; \mathbb{R}^{m_{W}})$ there exist unique $L_{q, \infty}$-solutions  to both equations in \eqref{eq:controlled_RSDE} in the sense of \cref{def.soln}. Further for both there is a $\mathfrak{C}_{T}^{\alpha}$-predictable version w.r.t. the filtration $(\mathfrak{F}_{s}')_{s \in [0,T]}$. 
    \end{theorem} 
    From now on we will always consider these versions and
     denote them by $(X^{ t, x,\pi; \mathbf{Z}}_{s})_{s \in [t,T]}, (\tilde{X}^{t, x, \pi; \mathbf{Z}}_{t})_{s \in [t,T]}$, respectively.
    \begin{proof}
        For the existence and uniqueness of solutions directly apply \cref{thm.fixpoint}. Of course this solution is then also a rough Itô process in the sense of \cite[Section 3]{flz25}. For the existence of the $\mathfrak{C}^{\alpha}_{T}$-predictable version, note that since $U \subset \mathbb{R}^{h}$ is compact, it is complete and separable under the Euclidean norm. Thus the product-space $U \times C([0,T]; \mathbb{R}^{m_{W}}) $ is complete and separable as well and so for any $\pi \in \mathcal{A}^{i}$ for any $i\in \{1,2, c\}$ the extension
        \begin{equation*}
            \pi^{\text{ext}} : [0,T] \times \Omega' \times \mathscr{C}^{0,\alpha, 1}([0,T]; \mathbb{R}^{m_{W}}) \to U \times C([0,T]; \mathbb{R}^{m_{W}}); \; \; (t, \omega', \mathbf{Z}) \mapsto \left(\pi_{t}(\omega'; \mathbf{Z}), Z_{\cdot \wedge t} \right)
        \end{equation*}
        is $\mathfrak{C}_{T}^{\alpha}$-predictable. Thus by measurability of the composition $\Theta \in \{b \circ \pi^{\text{ext}}, \sigma \circ \pi^{\text{ext}}\}$ and the pair $(\gamma, \gamma')$ together satisfy condition (A1) and (A2) in \cite[Section 3]{flz25} and thus by \cite[Theorem 3.2(i)]{flz25} there is a $\mathfrak{C}_{T}^{\alpha}$-predictable version of the solution process 
        \begin{equation*}
        [0,T] \times \Omega' \times \mathscr{C}^{0, \alpha, 1}([0,T]; \mathbb{R}^{m_{W}}) \ni (t, \omega', \mathbf{Z}) \mapsto X_{t}^{\pi, \mathbf{Z}}(\omega')
        \end{equation*}
        for any $\pi \in \mathcal{A}^{i}$, $i\in \{1,2, c\}$. 
    \end{proof}
    Next we want to connect the anticipative stochastic control problem \eqref{eq:path_dep_control_problem} with the corresponding rough stochastic control problem. To perform this 'freezing of the noise $W$' on the level of the control problem we need some further stability properties
    \begin{hypothesis}
    \label{ass:weaker_stability}
        Define for any $Z^{1}, Z^{2}\in \mathcal{C}^{\alpha}([0,T]; \mathbb{R}^{m_{W}})$
    \begin{equation*}
    \begin{aligned} 
       \max_{\varphi \in \{\tilde{b}, \sigma, \hat{l} \}} \sup_{u \in U} \sup_{t \in [0,T]}\Vert \varphi_{t}( \cdot, u, Z^1_{\cdot \wedge t})- \varphi_{t}( \cdot, u, Z^2_{\cdot \wedge t}) \Vert_{\infty}+ \Vert \hat{g}(\cdot, Z^1)- \hat{g}(\cdot, Z^2)\Vert_{\infty}&\eqqcolon A^{1}(Z^1; Z^2) \\
    \Vert \gamma_{t}(\cdot, Z_{\cdot \wedge t}^{1})- \gamma_{t}(\cdot,  Z^{2}_{\cdot \wedge t})\Vert_{\infty}\eqqcolon A^{2}(Z^{1}, Z^{2})
    \end{aligned}
    \end{equation*}
    and suppose that $A^{i}(Z^{n}, Z)\to 0$ for any sequence $(Z_{n})_{n \in \mathbb{N}}$ converging to $Z$ in $\mathcal{C}^{\alpha}$.
    \end{hypothesis}
    The following is an extension of \cite[Theorem 7.4]{flz26} using the good rough paths of \cref{subsection:good_rough_paths}.
    \begin{theorem}
    \label{thm:randomization_value_function}
        Suppose that Assumption \ref{progressive_assumption}, \ref{assumption_RSDE_existence_uniqueness} and \ref{ass:weaker_stability} hold. We define for any $i \in \{1,2, c\}$ the random value functions 
        \begin{equation*}
            \hat{V}^{i}_{t}(x) \coloneqq \essinf_{\pi \in \mathcal{A}^{i}}\mathbb{E} \left[\int_{t}^{T} \hat{l}_{s}(X_{s}^{t,x, \bar{\pi}}, \bar{\pi}_{s}) ds + \hat{g}(X_{T}^{t,x,\bar{\pi}}) \bigg|\mathfrak{F}^{W}_{T}\right]
        \end{equation*}
        as well as the rough value functions: 
        \begin{equation*}
        \begin{aligned}
        \mathcal{V}_{t}(x; \mathbf{Z})&\coloneqq \inf_{\pi \in \mathcal{A}} \mathbb{E}'\left[ \int_{t}^{T} \hat{l}_{s}(X_{s}^{t,x, \pi; \mathbf{Z}}, \pi_{s}, Z_{\cdot \wedge s}) ds + \hat{g}(X_{T}^{t,x,\pi;  \mathbf{Z}}) \right], \\
        \tilde{\mathcal{V}}_{t}(x; \mathbf{Z})&\coloneqq \inf_{\pi \in \mathcal{A}} \mathbb{E}'\left[ \int_{t}^{T} \hat{l}_{s}(\tilde{X}_{s}^{t,x, \pi; \mathbf{Z}}, \pi_{s}, Z_{\cdot \wedge s}) ds + \hat{g}(\tilde{X}_{T}^{t,x,\pi;  \mathbf{Z}}) \right].
        \end{aligned}
        \end{equation*} where $X^{\cdot;\mathbf{Z}}, \tilde{X}^{\cdot; \mathbf{Z}}$ are as in \eqref{eq:controlled_RSDE}; we also put  
        $\overline{\mathcal{V}}_{t}(x; \omega)\coloneqq \mathcal{V}_{t}(x; \mathbf{W}^{\text{It\^{o}}}(\omega''))= \tilde{\mathcal{V}}_{t}(x; \mathbf{W}^{\text{Strat}}(\omega''))$ for a.e. $\omega=(\omega', \omega'') \in \Omega$.
        Then it holds for any $(t, x) \in [0,T] \times \mathbb{R}^{d}$, a.s.
        \begin{equation*}
            \hat{V}^{1}_{t}(x)=\hat{V}^{2}_{t}(x)=\hat{V}^{c}_{t}(x)= \overline{\mathcal{V}}_{t}(x)= \hat{V}_{t}(x).
        \end{equation*}
    \end{theorem}
    \begin{proof}
        Note first, that $(\mathfrak{F}_{t})_{t \in [0,T]}$ is the same as the augmented filtration of the $m_{W}+ m_{B}$-dim. Brownian Motion $(W_{t}, B_{t})_{t \in [0,T]}$  on the space $(\Omega, \mathfrak{F}, \mathbb{P})$. Thus by \cref{appendix:lemma:optional_nonantic} for any $(\mathfrak{F}_{s})$-predictable control $(\pi_{s})_{s \in [0,T]}$ there is a Borel-measurable $\pi^{\text{det}}: [0, T]\times C([0,T]; \mathbb{R}^{m})\to U$ such that $\pi^{\text{det}}_{t}(B(\omega'), W(\omega''))=\pi^{\text{det}}_{t}(B_{\cdot \wedge t}(\omega'), W_{\cdot \wedge t}(\omega''))= \pi_{t}(\omega)$ for any $t \in [0,T]$, for $\mathbb{P}$-a.e. $\omega \in \Omega$. Thus denoting the set of such maps by  $\mathcal{U}^{\operatorname{bc}}$ and for $\pi \in \mathcal{U}^{bc}$, $\bar{\pi}_{t}=\pi_{t}(B_{\cdot \wedge t}, W_{\cdot \wedge t})$. Then for any $\pi \in \mathcal{U}^{bc}$ note that $\bar{\pi} \in \mathcal{A}^{c}$ and thus it holds for any $(t,x) \in [0,T] \times \mathbb{R}^{d}$ $\mathbb{P}$-a.s.
        \begin{equation*}
            \begin{aligned}
                \hat{V}^{1}_{t}(x) \geq \hat V_t (x)&= \essinf_{\pi \in {\mathcal U}}\mathbb E\left[ \int_t^T \hat l_s(X_s^{t,x,\pi},\pi_s) ds +\hat g(X_T^{t,x,\pi})   \bigg| \mathfrak{F}_T^W\right]\\
                &= \essinf_{\pi \in {\mathcal U}^{\operatorname{bc}}}\mathbb E\left[ \int_t^T \hat l_s(X_s^{t,x,\bar{\pi}},\bar{\pi}_s) ds +\hat g(X_T^{t,x,\bar{\pi}})   \bigg| \mathfrak{F}_T^W\right]\geq \hat{V}_{t}^{c}(x) \geq \hat{V}^{2}_{t}(x).
            \end{aligned}
        \end{equation*}
    From here on we omit the dependence on  $(t,x)$ for simplicity. The inequality $\bar{\mathcal{V}}\leq \hat V^{2}$ follows as in \cite[Theorem 7.4]{flz26}. 
    Thus it remains to show the inequality $\hat{V}^{1}\leq \bar{\mathcal{V}}$. First let $\tilde{W}:[0,T]\times \Omega'' \to \mathbb{R}^{d_{W}}$, $\mathfrak{B}([0,T])\otimes \mathfrak{F}^{''}_{T}/\mathfrak{B}^{d_{W}}$-measurable such that $\tilde{W}\in \mathcal{C}^{1}([0,T]; \mathbb{R}^{d_{W}})$ almost surely. We define for any $(Z, \tilde{Z})\in \mathcal{C}^{\alpha}\times \mathcal{C}^{1}$ the solution to the controlled SDE
    \begin{equation*}
    d\tilde{X}_{s}^{\pi; Z, \tilde{Z}}= \tilde{b}_{s} (\tilde{X}_{s}^{\pi; Z, \tilde{Z}}, \pi_{s}, Z_{\cdot \wedge s} ) ds + \sigma_{s} (\tilde{X}_{s}^{\pi; Z, \tilde{Z}}, \pi_{s},  Z_{\cdot \wedge s}) dB_{s}+ \gamma_{s}(\tilde{X}_{s}^{\pi; Z, \tilde{Z}},  Z_{\cdot \wedge s}) d\tilde{Z}_{s}; \quad \tilde{X}^{\pi; Z, \tilde{Z}}_{t}=x,  
    \end{equation*}
    where the latter integral is in Riemann-Stieltjes sense and $\pi \in \mathcal{A}$. Now let 
    \begin{equation*}
        J(\eta; Z, \tilde{Z})\coloneqq \mathbb{E}'\left[ \int_{t}^{T} \hat{l}_{s}(\tilde{X}_{s}^{ \pi; Z, \tilde{Z}}, \pi_{s}, Z_{\cdot \wedge s}) ds + \hat{g}(\tilde{X}_{T}^{\pi;  \mathbf{Z}}) \right]; \quad \tilde{\mathcal{V}}(; Z, \tilde{Z})\coloneqq \inf_{\eta \in \mathcal{A}} J( \eta; Z, \tilde{Z}). 
    \end{equation*}
    Now by Assumption \ref{assumption_RSDE_existence_uniqueness} and \ref{ass:weaker_stability} it follows that that both $\mathcal{V}, J$ are continuous in $(Z, \tilde{Z})\in \mathcal{C}^{\alpha}\times \mathcal{C}^{1}$ ($J$ even uniformly over $\eta \in \mathcal{A}$). Using this, one can perform an analogous argument to \cite[Theorem 7.4]{flz26} (ii) to show that $\tilde{\mathcal{V}}(W(\omega''), \tilde{W}(\omega''))\geq \hat V^{1, \tilde{W}}(\omega'')$ for $\mathbb{P}''$-a.e. $\omega'' \in \Omega''$, where $\bar{X}^{\pi\tilde{Z}}(\omega)\coloneqq \tilde{X}^{\pi;Z, \tilde{Z}}(\omega')|_{Z= W(\omega''), \tilde{Z}=\tilde{W}(\omega'')}$ and 
    \begin{equation*}
        \hat{V}^{1, \tilde{W}}(\omega'')\coloneqq \essinf_{\pi \in \mathcal{A}^{1}}\mathbb{E} \left[\int_{t}^{T} \hat{l}_{s}(\bar{X}_{s}^{ \bar{\pi}; \tilde{W}}, \bar{\pi}_{s}) ds + \hat{g}(\bar{X}_{T}^{\bar{\pi}}; \tilde{W}) \bigg|\mathfrak{F}^{W}_{T}\right](\omega'').
    \end{equation*}
    Now we know by \cref{appendix:ex:good_rp_Brownian} that $\mathbf{W}^{\operatorname{Strat}}$ a.s. has a piecewise linear good rough path sequence $\mathbf{W}^{n}$. Using then stability properties of $\mathcal{V}$ as in \cref{eq:good_rp_value_function} it follows for a.e. $\omega''\in 
    \Omega''$
    \begin{equation*}
        \bar{\mathcal{V}}(\omega'')=\mathcal{V}(\mathbf{W}^{\operatorname{Strat}}(\omega''))=\lim_{n\to \infty}\mathcal{V}(W(\omega''), W^{n}(\omega''))\geq \lim_{n \to \infty} \hat V^{1, W^{n}}(\omega'')\geq \hat{V}^{1}(\omega''), 
    \end{equation*}
    where for the last inequality we used that by \cite[Proposition 7.2]{flz26} and \cref{lemma:good_rp_RSDE} it holds
    \begin{equation*}
    \begin{aligned}
        |\hat{V}^{1}(\omega'')- \hat{V}^{1, W^{n}}(\omega'')|&\lesssim \esssup_{\pi \in \mathcal{A}^{1}}\left(\Vert \sup_{s\in [0,T]} |X_{s}^{\pi; Z, \tilde{Z}}- X_{s}^{\pi; (Z, \mathbb{Z}})\Vert_{L^{1}}\bigg|_{(Z, \mathbb{Z})=\mathbf{W}^{\operatorname{Strat}(\omega'')}, \tilde{Z}= W^{n}(\omega'')}\right)\\
        &\leq \rho_{\alpha}\left(\mathbf{W}^{\operatorname{Strat}(\omega'')}; (W^{n}(\omega''); \Pi(W; W^{n})(\omega''))\right) \stackrel{n \to \infty}{\rightarrow}0.
    \end{aligned}
    \end{equation*}
    \end{proof}
        Note that both $\mathcal{V}_{t}(x; \mathbf{Z}), \tilde{\mathcal{V}}_{t}(x; \mathbf{Z})$ are Markovian stochastic control problems with rough drivers and environments as the coefficients in \eqref{eq:path_dep_control_problem} only depend on the noise $(W_{t})_{t \in [0,T]}$, which is 'frozen' here and the controls $\pi \in \mathcal{A}$ only depend on the noise $B$. This way a rough dynamic programming principle and subsequently a corresponding rough HJB-equation can be recovered using the approach in \cite[Section 5 and 6]{flz26}. Let us focus on this approach next.
    \subsection{Viscosity solutions to the rough stochastic control problems}
    \label{subsection_Viscosity Solutions to the Rough Stochastic Control Problems}
    For the remainder of this section, we will exclusively focus on the rough stochastic control problem posed by $\tilde{\mathcal{V}}$. 
    Assuming suitable regularity of $\tilde{\mathcal{V}}$ one can, by applying the rough stochastic DPP \cite[Theorem 5.6]{flz26}, formally derive the corresponding 2nd-order rough HJB-Equation, assuming Stratonovich dynamics: 
    \begin{equation}
    \label{eq:Rough_HJB}
    \begin{aligned}
-d\tilde{\mathcal{V}}_{t}(x)&=\mathcal{H}_t( x ,D \tilde{\mathcal{V}}, D^2 \tilde{\mathcal{V}}, Z_{\cdot \wedge t}) d t+  D \tilde{\mathcal{V}}_{t}(x)\gamma_{t}(x, Z_{\cdot \wedge t}) d \mathbf{Z}_{t} \; \; \text{on} \; (0,T) \times \mathbb{R}^{d}\\
\tilde{\mathcal{V}}_T(x)&=\hat{g}(x, Z) \; \; \text{on} \; \mathbb{R}^{d}
\end{aligned}
\end{equation}
with the Hamiltonian $\mathcal{H}$ given by: 
\begin{equation*}
\begin{aligned}
&\mathcal{H}^{\text{cv}}_t(x, p,A ,u, z)\coloneqq \tilde{b}_t(x,u, z)\cdot p+\frac{1}{2}\operatorname{tr}[\sigma_t(\sigma_t)^{\top}(x,u, z)A]+\hat l_t(x,u, z)\\
&\mathcal{H}_t(x, p,A, z )\coloneqq\inf _{u \in U}\mathcal{H}^{\text{cv}}_t(x,p,A ,u, z).
\end{aligned}
\end{equation*}
As in the classical stochastic control framework, $\tilde{\mathcal{V}}$ will often not be sufficiently smooth to satisfy this equation in a strong sense, which leads naturally to the theory of \textit{viscosity solutions} for nonlinear rough PDEs (RPDEs) introduced in \cite{OberhauserCaruanaFriz2011}. 
\begin{definition}
\label{def:rough_viscosity_solution}
Let $\mathbf{Z} \in \mathscr{C}^{0,\alpha}_{g}([0,T]; 
    \mathbb{R}^{m_W})$. We say that $\tilde{\mathcal{V}}(\cdot, \cdot; \mathbf{Z})$ solves the equation \eqref{eq:Rough_HJB} if for any sequence of smooth paths $Z^{n}:[0,T]\to \mathbb{R}^{m_{W}}$ and canonical\footnote{by 'canonical' we mean $\mathbb{Z}^{n}_{s,t}\coloneqq \int_{s}^{t} \delta Z^{n}_{s,r} dZ^{n}_{r}$.} rough path lifts $\mathbf{Z}^{n}\coloneqq (Z^{n}, \mathbb{Z}^{n})_{n \in \mathbb{N}}$ with $\mathbf{Z}^{n}\to \mathbf{Z}$ in $\mathscr{C}^{0, \alpha}_{g}$ it holds $\tilde{\mathcal{V}}^{n} \to \tilde{\mathcal{V}}$ locally uniformly, where $\tilde{\mathcal{V}}(\cdot, \cdot; \mathbf{Z}^{n})$ is the unique viscosity solution (in the sense of Lions-Crandall) to 
    \begin{equation}
    \label{eq:smooth_HJB}
    \begin{aligned}
-\partial_t \tilde{\mathcal{V}}(x)&=\mathcal{H}_t( x ,D \tilde{\mathcal{V}}, D^2 \tilde{\mathcal{V}}, Z^{n}_{\cdot \wedge t}) +  D \tilde{\mathcal{V}}_{t}(x)\gamma_{t}(x, Z^{n}_{\cdot \wedge t}) \dot{Z^n_t} \; \; \text{on} \; (0,T) \times \mathbb{R}^{d}\\
\tilde{\mathcal{V}}_T(x)&=\hat{g}(x, Z) \; \; \text{on} \; \mathbb{R}^{d}. 
\end{aligned}
\end{equation}
\end{definition}
\begin{remark}
Note that in \cref{def:rough_viscosity_solution} we 'smooth out' the path-dependence on $Z$ everywhere in \eqref{eq:Rough_HJB}. At this point this is an arbitrary choice, after all one could also consider the limiting behavior of viscosity solutions to 
\begin{equation}
    \label{remark:eq:partial_smooth_Rough_HJB}
    \begin{aligned}
-\partial_{t}\tilde{\mathcal{V}}_{t}(x)&=\mathcal{H}_t( x ,D \tilde{\mathcal{V}}, D^2 \tilde{\mathcal{V}}, Z_{\cdot \wedge t}) + D\tilde{\mathcal{V}}_{t}(x)\gamma_t(x, Z_{\cdot \wedge t}) \dot{Z}^{n}_{t}  \; \; \text{on} \; (0,T) \times \mathbb{R}^{d},\\
\tilde{\mathcal{V}}_T(x)&=\hat{g}(x, Z) \; \; \text{on} \; \mathbb{R}^{d}. 
\end{aligned}
\end{equation}
Intuitively solutions to \eqref{remark:eq:partial_smooth_Rough_HJB} should even be 'closer' to the original value function $V$ as less quantities need to be approximated (after all by \cref{th:duality_stoch_viscosity} the rough path $\mathbf{Z}$ should be thought of as being 'given' to us). We will see in \cref{thm:good_rough_paths_viscosity} that this intuition is indeed correct. In general, it is not clear how $\gamma(\cdot, Z)$ can be seen as being controlled by the smooth path $Z^{n}$ (without resorting to imposing Young-regularity as in \cite{flz26}). We circumvent this issue by suitably employing the 'good' rough paths of \cite{coutin_2007} (see \cref{subsection:good_rough_paths}).
\end{remark}
In the following consider $\mathbf{Z}=(Z, \mathbb{Z}) \in \mathscr{C}^{0,\alpha}_{g}([0,T]; \mathbb{R}^{m_{W}})$ a fixed path. As we have seen in \cref{thm:randomization_value_function} it suffices to work on the filtered probability space $(\Omega' , \mathfrak{F}', (\mathfrak{F}'_{t})_{t\in [0,T]}, \mathbb{P}')$. Since we exclusively work with the measure $\mathbb{P}'$ for the remainder of this section,  all $L^{p}$-norms will be taken with respect to this measure i.e. $\Vert \cdot \Vert_{L^{p}}= \Vert \cdot \Vert_{L^{p}(\Omega'; \mathbb{P}')}$.\\
Since \cref{def:rough_viscosity_solution} is concerned with the stability of $\tilde{\mathcal{V}}$ under smooth approximations of the rough path, we need to impose the following stability properties in order to ensure stability of the controlled RSDE and the control problem itself.

\begin{hypothesis}    \label{ass:rpde_viscos_1}
   Suppose that for 
   \begin{equation*}
   \Theta(t,x, u, Z_{\cdot \wedge t})  \in \{\hat{l}(t,x, u, Z_{\cdot \wedge t}), \tilde{b}(t,x, u, Z_{\cdot \wedge t}), \gamma(t,x,  Z_{\cdot \wedge t}), \sigma(t,x, u, Z_{\cdot \wedge}), \hat{g}(x, Z)\}
   \end{equation*}
   it holds $
       [0,T] \times \mathbb{R}^{d} \times U \ni (t, x, u) \mapsto \Theta(t,x, u, Z_{\cdot \wedge t})$
   is uniformly continuous.
   Further suppose that there is $L>0$ such that 
   \begin{equation*}
       |\Theta(t, x, u)- \Theta(t,y, u)| \leq L |x-y|;\quad |\Theta(t, 0, u)| \leq L 
   \end{equation*}
   for any $t \in [0,T], x,y \in \mathbb{R}^{d}$ and $u \in U$.
\end{hypothesis}
\begin{hypothesis}
    \label{ass:rpde_viscos_2}
    Define for any $Z^{1}, Z^{2}\in \mathcal{C}^{\alpha}([0,T]; \mathbb{R}^{m_{W}})$
    \begin{equation*}
    \begin{aligned} 
    \sup_{t\in [0,T]}\Vert \gamma_{t}(\cdot, Z^{1})-\gamma_{t}(\cdot, Z^{2})\Vert_{C^{\beta}_{b}}+\sup_{t \in [0,T]}\Vert \gamma'_{t}(\cdot, Z^{1})-\gamma'_{t}(\cdot, Z^{2})\Vert_{C^{\beta-1}_{b}} &\eqqcolon A^{3}(Z^1; Z^2)\\
    [ \gamma(\cdot, \cdot, Z^{1}), \gamma'(\cdot, \cdot, Z^{1}); \gamma(\cdot, \cdot, Z^{2}), \gamma'(\cdot, \cdot, Z^{2})]_{Z^{1}, Z^{2}; 2\alpha} &\eqqcolon A^{4}(Z^1; Z^2)
    \end{aligned}
    \end{equation*}
    with the latter notation as in \cref{def:scvec}. Then assume that for any smooth sequence $(Z^{n})_{n \in \mathbb{N}}$ with $Z^{n}\to Z$ in $\mathcal{C}^{\alpha}$ it holds
        $\lim_{n\to \infty}A^{i}(Z; Z^{n})=0$ for any $i={1,3, 4}$ where $A^{1}$ is as in \cref{ass:weaker_stability}.
\end{hypothesis}
We first consider the well-posedness of \eqref{eq:Rough_HJB} in the 'classical' sense of \cref{def:rough_viscosity_solution}.
\begin{theorem}
\label{thm:rough_viscosity_classic}
    Suppose \cref{assumption_RSDE_existence_uniqueness}, \cref{ass:rpde_viscos_1} and \cref{ass:rpde_viscos_2} hold. Then $\tilde{\mathcal{V}}(t,x; \mathbf{Z})$ is the unique viscosity solution to \eqref{eq:Rough_HJB} in the sense of \cref{def:rough_viscosity_solution}. In particular for any sequence of smooth paths $(Z^{n})_{n \in \mathbb{N}}$ with canonical rough path lift $\mathbf{Z}^{n}\to \mathbf{Z}$ in $\mathscr{C}^{0, \alpha}_{g}$, \eqref{eq:smooth_HJB} has a unique viscosity solution given by $\tilde{\mathcal{V}}(\cdot, \cdot; \mathbf{Z}^{n})$ and 
    \begin{equation*}
    \begin{aligned}
        \sup_{t\in [0,T], x \in \mathbb{R}^{d}}|\tilde{\mathcal{V}}(t,x; \mathbf{Z})- \tilde{\mathcal{V}}(t,x;\mathbf{Z}^{n})|\lesssim & \sum_{i=1,3,4} A^{i}(Z; Z^{n})+ \rho_{\alpha}(\mathbf{Z}; \mathbf{Z}^{n})\stackrel{n \to \infty}{\longrightarrow} 0
    \end{aligned}
    \end{equation*}
\end{theorem}
\begin{proof}
    Under the above assumptions \cite[Chapter 4, Theorem 5.2 and Theorem 6.1]{YongZhou1999} assert that $\tilde{\mathcal{V}}(\cdot, \cdot; \mathbf{Z}^{n})$ is the unique viscosity solution in the sense of Lions-Crandall of \eqref{eq:smooth_HJB}. The convergence to $\tilde{\mathcal{V}}$ now follows by noting that 
    \begin{equation*}
    \begin{aligned}
        &\big| \tilde{\mathcal{V}}_{t}(x; \mathbf{Z}) - \tilde{\mathcal{V}}_{t}(x; \mathbf{Z}^{n}) \big|\\
        &\lesssim \sup_{\pi \in \mathcal{A}} \bigg|\mathbb{E}'\left[ \int_t^T \hat{l}_s(X_s^{t,x,\pi, \mathbf{Z}},\pi_s, Z_{\cdot \wedge s})-\hat l_s(\tilde{X}_s^{t,x,\pi, \mathbf{Z}^{n}},\pi_s, Z^n_{\cdot \wedge s}) ds +\hat g(\tilde{X}_T^{t,x,\pi, \mathbf{Z}}, Z)- \hat{g}(\tilde{X}_T^{t,x,\pi, \mathbf{Z}^{n}}, Z^n) \right] \bigg| \\
        &\lesssim
        \int_{0}^{T} \sup_{u \in U}  \Vert \hat{l}(s,  \cdot, u, Z_{\cdot \wedge s})- \hat{l}(s,  \cdot, u, Z^n_{\cdot \wedge s}) \Vert_{\infty} ds +\Vert \hat{g}(\cdot, Z)- \hat{g}(\cdot, Z^n)\Vert_{\infty}\\
        &\qquad+ \sup_{\pi \in \mathcal{A}} \left \Vert \sup_{s \in [t,T]} |\tilde{X}^{t, x, \pi; \mathbf{Z}}_{s}- \tilde{X}^{t,x, \pi; \mathbf{Z}^{n}}_{s} | \right \Vert_{L^{1}}. 
    \end{aligned}
    \end{equation*}
     The latter term is vanishing, uniformly in $(t, x)$, by using the local Lipschitz continuity of solutions to RSDEs from \cite[Theorem 4.9]{fhl21} combined with the stability estimates from \cref{ass:rpde_viscos_1} and \cref{ass:rpde_viscos_2}.
\end{proof}
While the above approach is classical for viscosity solutions of RPDEs note that many stochastic control problems of interest will not satisfy the continuous dependence on the noise $W$ (resp. $Z$) assumed in \cref{ass:rpde_viscos_2}. This motivates the following result, which avoids \cref{ass:rpde_viscos_2}. 
\begin{theorem}
\label{thm:good_rough_paths_viscosity}
    Suppose \cref{assumption_RSDE_existence_uniqueness} and \cref{ass:rpde_viscos_1} hold. Further assume that there is a piece-wise linear sequence of good rough paths $(Z^{n})_{n \in \mathbb{N}}$ w.r.t. $\mathbf{Z}$ in the sense of \cref{def:good_rough_path}. Then for any $n\in \mathbb{N}$ denote by $\tilde{X^{t, x, \pi, n}}$ the unique solution of the controlled SDE 
    \begin{equation*}
    \begin{aligned}
    d\tilde{X}_{s}^{t, x, \pi, n}= &\tilde{b}_{s} (\tilde{X}_{s}^{t, x, \pi, n}, \pi_{s}, Z_{\cdot \wedge s} ) ds + \sigma_{s} (\tilde{X}_{s}^{t, x, \pi,n }, \pi_{s},  Z_{\cdot \wedge s}) dB_{s}+ \gamma_{s}(\tilde{X}_{s}^{t, x, \pi, n},  Z_{\cdot \wedge s}) dZ^{n}_{s}; \quad \tilde{X}^{t, x, \pi, n}_{t}=x 
    \end{aligned}
    \end{equation*}
    Then for any $n \in \mathbb{N}$ it holds 
    \begin{equation*}
        \tilde{\mathcal{V}}^{n}_{t}(x)\coloneqq \inf_{\pi \in \mathcal{A}} \mathbb{E}'\left[\int_{t}^{T} \hat{l}^{h}(\tilde{X}^{t,x, \pi, n}_s, \pi_s, Z_{\cdot \wedge s}) ds + \hat{g}(\tilde{X}^{t,x, \pi, n}_{T}, Z)\right]
    \end{equation*}
    is the unique viscosity solution to \begin{equation}
    \label{eq:partial_smooth_Rough_HJB}
    \begin{aligned}
-\partial_{t} \tilde{\mathcal{V}}_{t}(x)&= \mathcal{H}_{t}( x ,D \tilde{\mathcal{V}}, D^2 \tilde{\mathcal{V}}, Z_{\cdot \wedge t}) +D \tilde{\mathcal{V}}_{t}(x)\gamma_{t}(x, Z_{\cdot \wedge t})  \dot{Z}^{n}_{t}  \; \; \text{on} \; (0,T) \times \mathbb{R}^{d},\\
\tilde{\mathcal{V}}_T(x)&=\hat{g}(x, Z) \; \; \text{on} \; \mathbb{R}^{d}. 
\end{aligned}
\end{equation}
   Further it holds with the iterated Riemann-Stieltjes integral $\Pi(Z; Z^{n})_{s,t} \coloneqq \int_{s}^{t} \delta Z_{s, r} dZ^{n}_{r}$:
   \begin{equation*}
       \sup_{t\in [0,T], x \in \mathbb{R}^{d}}\big| \tilde{\mathcal{V}}^{n}_{t}(x)- \tilde{\mathcal{V}}_{t}(x; \mathbf{Z})\big| \lesssim \rho_{\alpha}(\mathbf{Z}; (Z^{n}, \Pi(Z; Z^{n})))\stackrel{n \to \infty}{\longrightarrow} 0.
   \end{equation*}
\end{theorem}
\begin{proof}
    Since ${Z}^{n}$ is piece-wise linear we can choose a partition $0=t_{0}< t_{1} < \dots <t_{N}=T$ such that $\dot{Z}^{n, +}$ is constant on $[t_{i}, t_{i+1})$ for any $i=0, \dots, N-1$\footnote{We consider here the \emph{one-sided derivative} $\dot{Z}^{n, +}_{t}\coloneqq \lim_{h \downarrow 0}\frac{Z^{n}_{t+h}-Z^{n}_{t}}{h}$. On $(t_{i}, t_{i+1})$ this aligns of course with the usual derivative $\dot{Z}^{n}_{t}$.}. Setting $\tilde{\mathcal{V}}^{N, n}_{T}\equiv g$ we see that for any $i=0, \dots, N-1$ again by \cite[Chapter 4, Theorem 5.2 and Theorem 6.1]{YongZhou1999} the PDE
    \begin{equation*}
    \begin{aligned}
        -\partial_{t} \tilde{\mathcal{V}}^{i, n}_{t}(x)&= \mathcal{H}_t( x ,D \tilde{\mathcal{V}}^{i, n}, D^2 \tilde{\mathcal{V}}^{i,n}, Z_{\cdot \wedge t}) +(\gamma_{t}(x, Z_{\cdot \wedge t}) \cdot D \tilde{\mathcal{V}}^{i, n}_{t}(x)) \dot{Z}^{n}_{t}  \; \; \text{on} \; (t_{i},t_{i+1}) \times \mathbb{R}^{d},\\
\tilde{\mathcal{V}}^{i, n}_{t_{i+1}}&= \tilde{\mathcal{V}}^{i+1, n}_{t_{i+1}}\quad \text{on} \; \mathbb{R}^{d}
\end{aligned}
    \end{equation*}
    has a unique viscosity solution given for any $t \in [t_{i}, t_{i+1})$ by
    \begin{equation*}
        \tilde{\mathcal{V}}^{i,n}_{t}(x)\coloneqq \inf_{\pi \in \mathcal{A}} \mathbb{E}'\left[\int_{t}^{t_{i+1}} \hat{l}^{h}(\tilde{X}^{t,x,\pi, n}_s, \pi_s, Z_{\cdot \wedge s}) ds + \tilde{\mathcal{V}}^{i+1, n}_{t_{i+1}}(\tilde{X}_{t_{i+1}}^{t,x,\pi, n})\right].
    \end{equation*}
By the DPP we see that $\tilde{\mathcal{V}}^{0, n}\equiv \tilde{\mathcal{V}}^{n}$ and by the local nature of viscosity solutions we see that thus $\tilde{\mathcal{V}}^{n}$ is the unique viscosity solution to \eqref{eq:partial_smooth_Rough_HJB} on $[0,T]\times \mathbb{R}^{d}$. The convergence of $\tilde{\mathcal{V}}^{n}$ to $\tilde{\mathcal{V}}$ is due to  \cref{ass:rpde_viscos_1} as following the argument given in the proof of \cref{thm:rough_viscosity_classic} and then \cref{lemma:good_rp_RSDE}, we have
\begin{equation}
\label{eq:good_rp_value_function}    |\tilde{\mathcal{V}}^{n}_{t}(x)- \tilde{\mathcal{V}}_{t}(x; \mathbf{Z})|\lesssim \sup_{\pi \in \mathcal{A}} \left \Vert \sup_{s \in [t,T]} \big|\tilde{X}^{t, x,\pi, n }_{s}- \tilde{X}^{t,x, \pi; \mathbf{Z}}_{s} \big| \right \Vert_{L^{1}}\lesssim \rho_{\alpha}( \mathbf{Z}; (Z^{n}, \Pi(Z; Z^{n}))) \stackrel{n \to \infty}{\longrightarrow} 0,
\end{equation}
where all implicit constants are independent of $Z^{n}$. 
\end{proof}
\begin{remark}
    Note that by \cref{appendix:ex:good_rp_Brownian}, \cref{thm:good_rough_paths_viscosity} is applicable in the context of \cref{thm:randomization_value_function} given \cref{ass:rpde_viscos_1} holds. 
\end{remark}
A natural numerical scheme for solving \eqref{eq:Rough_HJB} would be to approximate either $\tilde{\mathcal{V}}^{n}$ or $\tilde{\mathcal{V}}(\cdot, \cdot; \mathbf{Z}^{n})$ by classical numerical schemes for suitably large $n \in \mathbb{N}$. This methodology was considered in \cite[Section 4]{bank_rough_2025} for linear RPDEs. We emphasize that approximating $\tilde{\mathcal{V}}^{n}$ is clearly advantageous compared to approximating $\tilde{\mathcal{V}}(\cdot, \cdot; \mathbf{Z}^{n})$ as seen in the following example. 
\begin{example}
\label{example:benefit_good_approximation}
    Let $\varphi \in C(\mathbb{R}^{d}\times \mathbb{R}^{m_{W}}; \mathcal{L}(\mathbb{R}^{m_{W}}; \mathbb{R}^{d}))$ such that $\varphi (\cdot, z)\in \operatorname{Lip}^{3}(\mathbb{R}^{d}; \mathcal{L}(\mathbb{R}^{m_{W}}; \mathbb{R}^{d}))$ uniformly over $z \in \mathbb{R}^{m_{W}}$ and $\varphi(x, \cdot) \in \operatorname{Lip}^{2}(\mathbb{R}^{m_{W}}; \mathcal{L}(\mathbb{R}^{m_{W}}; \mathbb{R}^{d}))$ uniformly over $x \in \mathbb{R}^{d}$. Then we define $\gamma_{t}(x, Z)\coloneqq \varphi(x, Z_{t})$ for any $Z\in \mathcal{C}^{\alpha}([0,T]; \mathbb{R}^{m_{W}})$ $(t, x)\in [0,T]\times \mathbb{R}^{d}$. Then for $\gamma'_{t}(x, Z)\coloneqq D_{z} \varphi(x, Z_{t})$ we immediately see by classical results on the composition of controlled rough paths with regular functions \cite[Chapter 7]{friz_2020}, that $(\gamma, \gamma')$ satisfy \cref{assumption_RSDE_existence_uniqueness} and\cref{ass:rpde_viscos_1} and thus \cref{thm:good_rough_paths_viscosity} is applicable under good rough paths assumptions and suitable assumptions on the other involved coefficients.  In order to also satisfy \cref{ass:rpde_viscos_2}, however, we would need to impose additional regularity of $\varphi$ for instance $\varphi(x, \cdot) \in \operatorname{Lip}^{2+\epsilon}(\mathbb{R}^{m_{W}}; \mathcal{L}(\mathbb{R}^{m_{W}}; \mathbb{R}^{d}))$ uniformly over $x$ for some $\epsilon> 0$. This would yield, following the proof of \cite[Theorem 7.6]{friz_2020}, a modulus of continuity 
    \begin{equation*}
        [ \gamma(\cdot, \cdot, Z^{1}), \gamma'(\cdot, \cdot, Z^{1}); \gamma(\cdot, \cdot, Z^{2}), \gamma'(\cdot, \cdot, Z^{2})]_{Z^{1}, Z^{2}; 2\alpha} \lesssim \Vert Z^{1}- Z^{2}\Vert_{\alpha}^{\epsilon}.
    \end{equation*}
    Under Brownian randomization this would in the context of \cref{thm:rough_viscosity_classic} and \cref{thm:good_rough_paths_viscosity} yield convergence rates of the form 
    \begin{equation*}
    \begin{aligned}
        \left \Vert  \sup_{t\in [0,T], x \in \mathbb{R}^{d}}\big| \tilde{\mathcal{V}}^{n}_{t}(x)- \tilde{\mathcal{V}}_{t}(x; \mathbf{W}^{\text{Strat}})\big| \right \Vert_{L^{2}}  &
    \lesssim R(n)^{-1}; \\
        \left \Vert  \sup_{t\in [0,T], x \in \mathbb{R}^{d}}\big| \tilde{\mathcal{V}}_{t}(x; (\mathbf{W}^{\text{Strat}})^{n})- \tilde{\mathcal{V}}_{t}(x; \mathbf{W}^{\text{Strat}})\big| \right \Vert_{L^{2}}   &\lesssim R(n)^{-\epsilon},
        \end{aligned}
        \end{equation*}
        where $(\mathbf{W}^{\text{Strat}})^{n}$, the piece-wise linear approximation of (enhanced) Brownian motion over a sequence of partitions $(\mathcal{P}^{n})$ of $[0,T]$, is the good rough path sequence considered here. Here $R(n)$ is the usual error term of such approximations see \cite[Lemma 4]{coutin_2007}. Therefore leveraging the good rough paths structure yields \emph{better rates of convergence} while imposing \emph{less regularity assumptions} on $\gamma$. 
\end{example}
To close off this section let us clarify, why  \cref{ass:rpde_viscos_1} is reasonable in the setting of \cref{th:duality_smooth_semimartingale} in an approximative sense: 
\begin{lemma}
    Suppose the assumptions of \cref{th:duality_smooth_semimartingale} are satisfied. Then for any $h \in \mathcal{H}$ with components $(\mathfrak{d}_{t}h, \mathfrak{d}_{\omega}h)$ there are sequences $(a^{n})_{n}, (b^{n})_{n}, (c^{n})_{n},(d^{n})_{n} \in \mathcal{S}^{2}(\Lip(\mathbb{R}^{d}))$  such that it holds almost surely 
    \begin{equation}
    \label{lemma:approximation_convergence}
        \mathbb{E}\left [\hat{V}_{t}(x; Dh, D^{2}h, \mathfrak{d}_{t} h , D(\mathfrak{d}_{\omega} h))|\mathfrak{F}^{W}_{t} \right]= \lim_{n \to \infty} \mathbb{E}\left [\hat{V}_{t}(x; a^{n}, b^{n}, c^{n}, d^{n})|\mathfrak{F}^{W}_{t} \right], 
    \end{equation}
    where 
    \begin{equation*}
    \hat V_t(x;a,b,c,d):= \essinf_{\pi \in {\mathcal U}}\mathbb E\left[ \int_t^T \hat l_s(X_s^{t,x,\pi},\pi_s; a,b,c,d) ds   \bigg| \mathfrak{F}_T^W\right], 
  \end{equation*}
    where $\hat l_s(x,\pi; a,b,c,d)\coloneqq c_{s}(x) +H_s^{\text{cv}}(x,d_s(x), a_s(x) , b_s(x),\pi)$
    In particular note that assuming further that \cref{assumption_RSDE_existence_uniqueness} holds, the coefficients $\{\tilde{b}, \sigma, \gamma, \hat{l}(\cdot; a^{n}, b^{n}, c^{n}, d^{n})\}$ satisfy \cref{ass:rpde_viscos_1} for any $n \in \mathbb{N}$. 
\end{lemma}
\begin{proof}Let $n \in \mathbb{N}$ and let $B_{n}\subset \mathbb{R}$ be the centered ball of radius $n$. Then by \cref{approx_coroll_smooth_semimart} \eqref{eq:appendix_convergence} we can choose $(a^{n})_{n} \subset \mathcal{S}^{2}(\Lip(\mathbb{R}^{d}))$ such that $\Vert Dh-a^{n}\Vert_{C(B_{n})}\leq \frac{1}{n}$ and $\Vert a^{n}\Vert_{C_{b}(\mathbb{R}^{d})}\leq \Vert h \Vert_{\mathcal{S}^{2}(C_{b}^{2}(\mathbb{R}^{d}))}+ \nicefrac{1}{n}$. Choosing analogously $b^{n},c^{n},d^{n} \in \mathcal{S}^{2}(\Lip(\mathbb{R}^{d}))$ we follow that on $A_{n}\coloneqq \{ \esssup_{\pi} \sup_{s \in [0,T]} |X^{t, x, \pi}_{s}|\in B_{n} \}$ it holds
\begin{equation*}
    \mathbb{E} \left [ \bigg| \hat{V}_{t}(x; Dh, D^{2}h, \mathfrak{d}_{t} h , D(\mathfrak{d}_{\omega} h))- \hat{V}_{t}(x; a^{n}, b^{n}, c^{n}, d^{n})\bigg|\mathbbm{1}_{A_{n}}\right ]\leq \frac{1}{n}
\end{equation*}
and on the complement it holds 
\begin{equation*}
   \mathbb{E} \left [ \bigg| \hat{V}_{t}(x; Dh, D^{2}h, \mathfrak{d}_{t} h , D(\mathfrak{d}_{\omega} h))- \hat{V}_{t}(x; a^{n}, b^{n}, c^{n}, d^{n})\bigg|\mathbbm{1}_{A_{n}^{c}}\right ]\lesssim \frac{1}{n} + \mathbb{P}(A_{n}^{c}),
\end{equation*}
where the implicit constant depends only on global bounds on $h$ and it's components i.e. $\Vert h \Vert_{\mathcal{S}^{2}(C^{2}_{b}(\mathbb{R}^{d}))}$ etc. Taking $n \to \infty$ and using $\mathbb{P}(A_{n}^{c})\to 0$ yields \eqref{lemma:approximation_convergence} in $L^{1}$-sense and thus convergence almost surely upon taking a suitable subsequence. The coefficients depending on $(a^{n}, \dots, d^{n})$ being sufficiently regular to satisfy \cref{ass:rpde_viscos_1} in the case of given \cref{assumption_RSDE_existence_uniqueness} is a direct consequence of boundedness and random bounded Lipschitzness.
\end{proof}
An analogous claim in the context of \cref{th:duality_stoch_viscosity} can be shown by additionally approximating the penalized terminal cost $g-h_{T}$.
 \appendix
 \addtocontents{toc}{\protect\setcounter{tocdepth}{1}}
\section{Notation}\label{notations}
\subsection{Random variables.} For any topological space $V$ we denote by $\mathfrak{B}(V)$ the corresponding Borel-$\sigma$-algebra. Let $(\Omega, \mathfrak{F}), (A, \mathfrak{A})$ be two measurable spaces. We say a map $X: \Omega \to A$ is $\mathfrak{F}/\mathfrak{A}$ measurable if for any $B \in 
\mathfrak{A}$, $X^{-1}(B)\in \mathfrak{F}$. Let $(\Omega, \mathfrak{F},(\mathfrak{F}_t)_{t \in [0,T]}, \mathbb{P})$ be a complete filtered probability space carrying an $m$-dimensional Wiener process $W=\left\{W_t: t \in[0, T]\right\}$  such that $(\mathfrak{F}_t)_{t \in [0,T]}$ is the natural filtration generated by $W$ and augmented by all the $\mathbb{P}$-null sets in $\overline{\mathfrak{F}}$. We denote by $\mathfrak{P}$ the $\sigma$-algebra of the predictable sets on $[0, T] \times \Omega$ associated with $(\mathfrak{F}_t)_{t \geq 0}$. 
 \subsection{Function spaces}
Let $(B,\|\cdot\|_B)$ be a Banach space. For a sub-$\sigma$-algebra
$\mathfrak{G}\subseteq\mathfrak{F}$, we say $X: \Omega \to B$ is strongly $\mathfrak{G}/\mathfrak{B}(B)$-measurable, if it is $\mathfrak{G}/\mathfrak{B}(B)$-measurable and separably valued.  For $p \in[1, \infty]$, $\mathcal{S}^p(\mathbb{B})$ is the set of all the $\mathbb{B}$-valued, $\mathfrak{P}$-measurable continuous processes $\left\{\mathcal{X}_t\right\}_{t \in[0, T]}$ such that
$
\|\mathcal{X}\|_{\mathcal{S}^p(\mathbb{B})}:=\left\|\sup _{t \in[0, T]}\left\| \mathcal{X}_t\right\|_{\mathbb{B}}\right\|_{L^p(\Omega, \mathfrak{F}, \mathbb{P})}<\infty
$. Denote by $\mathcal{L}^p(\mathbb{B})$ the space of all the $\mathbb{B}$-valued, $\mathfrak{P}$-measurable processes $\left\{\mathcal{X}_t\right\}_{t \in[0, T]}$ such that
\begin{equation*}
\|\mathcal{X}\|_{\mathcal{L}^p(\mathbb{B})}:=\left\|\left(\int_0^T\left\|\mathcal{X}_t\right\|_{\mathbb{B}}^p d t\right)^{1 / p}\right\|_{L^p(\Omega, \mathfrak{F}, \mathbb{P})}<\infty.
\end{equation*}
For $k\in\mathbb N_0, l \in \mathbb{N}$, let $C^k(\mathbb R^d; \mathbb{R}^{l}),C_b^k(\mathbb R^d; \mathbb{R}^{l})$ denote the space of $\mathbb{R}^{l}$-valued 
$k$-times continuously differentiable functions and $k$-times continuously differentiable functions whose derivatives up
to order $k$ are bounded. When $l=1$ we simply write $C^k(\mathbb R^d),C_b^k(\mathbb R^d)$. spatial derivatives are denoted by $D$.
We further set
$C_0(\mathbb R^d)
 :=
 \{f\in C_b(\mathbb R^d): f(x)\to0
       \text{ as }|x|\to\infty\}.$
 For each $(k, q) \in \mathbb{N}_0 \times[1, \infty]$ we denote by $H^{k, q}=H^{k, q}(\mathbb R^d)$ the $k$ the Sobolev space. $\operatorname{BUC}(\mathbb{R}^{d})$ denotes the space of $\mathbb{R}$-valued bounded uniformly continuous functions on $\mathbb{R}^{d}$. The space $\Lip(\mathbb{R}^{d}; \mathbb{R}^{l})$ denotes the $\mathbb{R}^{l}$-valued bounded and Lipschitz-continuous functions, with analogous convention for $l=1$ as above. We set $\Vert \varphi \Vert_{\Lip}\coloneqq \Vert \varphi \Vert_{\infty}+ \sup_{x \neq y}\frac{|\varphi(x)-\varphi(y)|}{|x-y|}$. We say $\varphi \in \Lip^{\gamma}$ for $\gamma= n+ \beta$ with $n\in \mathbb{N}, \beta \in (0,1]$ if $\varphi \in C^{n}$ and $D^{n}\varphi \in \mathcal{C}^{\beta}$ $\beta$-Hölder continuous. 
\subsection{Path spaces:}
Consider a fixed time-horizon $T>0$. Let
$
\Delta_T \coloneqq \{(s,t): 0 \le s < t \le T \}.
$
Let $\gamma \in (0, \infty)$. We say that $X \in \mathcal{C}^{\gamma}([0,T];V)$ if it
is $\lfloor \gamma \rfloor$ times continuously differentiable with $\lfloor
\gamma \rfloor$-th derivative H\"older continuous of exponent $\{ \gamma \} =
\gamma - \lfloor \gamma \rfloor \in (0, 1]$ (In particular, $X \in
\mathcal{C}^1$ means Lipschitz rather than continuously differentiable). For $\gamma \in (0,1]$ we denote by $\Vert X \Vert_{\gamma; [0,T]}$ the $\gamma$-Hölder semi-norm on $[0,T]$. For $A: \Delta_{T}\to V$ and $\gamma \in (0, \infty)$ we say that $A\in \mathcal{C}^{\gamma}_{2}([0,T]; V)$ if $\sup_{(s,t) \in \Delta_{T}}\frac{|A_{s,t}|}{|t-s|^{\gamma}}< \infty$. For a path $X: [0,T]\to V$ we denote by $\delta X_{s,t}\coloneqq X_{t}-X_{s}$ and for $A:\Delta_{T}\to V$, $\delta A_{s,u,t}\coloneqq A_{s,t}-A_{s,u}-A_{u,t}$.
\

\subsection{Rough paths} 
Here we we revisit for the convenience of the reader notation from classical (deterministic) rough path theory following \cite{friz_2020}. For an exposition on \textit{rough stochastic analysis} we refer to \cref{appendix_section_rsde_control}. 
Let $V$ be a finite-dim. Banach space. $\otimes$ denotes the tensor product. We denote the space of $V$-valued rough paths $\mathbf{X}\coloneqq (X, \mathbb{X})$ on $[0,T]$ by $\mathscr{C}^{ \alpha}([0,T]; V)$. $$\mathbf{X}=(X, \mathbb{X})\in \mathscr{C}^{\alpha}([0,T]; V), \qquad 
[\mathbf{X}] := (\delta X) \otimes (\delta X) - 2 \operatorname{Sym}(\mathbb{X}).$$
We denote by $\mathscr{C}_g^{\alpha}= \{ \mathbf{X} \in \mathscr{C}^{\alpha}| [\mathbf{X}]\equiv 0 \}$, the space of weakly geometric rough paths, and by $\mathscr{C}^{0, \alpha}_{g}$ we denote the space of geometric rough paths. We denote by $\mathscr{C}^{0, \alpha}$ the Polish space of $\alpha$-Hölder rough paths obtained as $\rho_{\alpha}$-closure (see below) of smooth rough paths.
 We denote by $\mathscr{C}^{0, \alpha,1}\coloneqq \{\mathbf{X}\in \mathscr{C}^{0, \alpha}: \Vert[\mathbf{X}]\Vert_{C^{1}}< \infty \}$ the (Polish) space of rough paths with continuously differentiable bracket and by $\mathfrak{C}^{\alpha}_{T}$ its Borel-$\sigma$-algebra.
The space of   
$X$-controlled $W$-valued rough paths is denoted by $\mathscr{D}^{2\alpha}_{X}([0,T]; W)$. We define the \textit{step-2 tensor space} $T^{2}\coloneqq V \oplus V^{\otimes 2}$. For two $T^{2}$-valued pairs of processes $(X, \mathbb{X}), (Y, \mathbb{Y})$ where $X, Y$ are $V$-valued paths and $\mathbb{X}, \mathbb{Y}: \Delta_{T}\to V^{\otimes 2}$ we define
\begin{equation*}
\begin{aligned}
    \rho_{\alpha}((X, \mathbb{X}), (Y, \mathbb{Y}))&\coloneqq \sup_{(s,t)\in \Delta_{T}} \frac{|\delta X_{s,t}- \delta Y_{s,t}|}{|t-s|^{\alpha}}+ \sup_{(s,t)\in \Delta_{T}} \frac{|\mathbb{X}_{s,t}- \mathbb{Y}_{s,t}|}{|t-s|^{2\alpha}}.
\end{aligned}
\end{equation*}
For a $d$-dim. Brownian Motion $W$ we denote by $\mathbf{W}^{\text{It\^{o}}}=(W, \mathbb{W}^{\text{It\^{o}}})$, $\mathbf{W}^{\text{Strat}}=(W, \mathbb{W}^{\text{Strat}})$ the \textit{It\^{o}}- resp. \textit{Stratonovich-lift} of $W$, where $\mathbb{W}^{\text{It\^{o}}}_{s,t}\coloneqq \int_{s}^{t} \delta W_{s,r} \otimes dW_{r}, \mathbb{W}^{\text{Strat}}_{s,t}\coloneqq \int_{s}^{t} \delta W_{s,r} \otimes \circ dW_{r}$.
\section{Duality for the Markovian case}\label{sec: duality_markvian_appendix}
In this section, we  simplify the setting of the paper to doubly controlled \emph{Markovian} SDEs and provide a generalization of the approach in \cite{DiehlFrizGassiat2017,bank_duality_2026} to this setting. {This allows us to apply the arguments of the present paper in the standard case, without the restrictions due to the presence of  random coefficients (i.e. \cref{hp:regularity_coeff}).}

Consider a controlled Markovian SDE 
$$dX_s^{t, x, \pi}=b_{s}(X_s^{t, x, \pi},\pi_s)ds+\sigma_{s} (X_s^{t, x, \pi},\pi_s)dW_s+\gamma_{s} (X_s^{t, x, \pi},\pi_s)dB_s , \quad X_t^{t, x, \pi}=x\in \mathbb R^d,$$
where $W,B$ are independent Brownian motions, $b, \sigma, \gamma$ are deterministic but otherwise as in \cref{sec:duality} and $\pi\in \mathcal U=\{\pi \colon [0, T]\times \Omega \to U \textit{ is }(\mathfrak{F}_s) \textit{-predictable} \}$ are admissible controls. Define the value function
$$V_{t}(x):=\inf_{\pi \in \mathcal U}\mathbb E\left[\int_t^T l_{s}(X_s^{t,x,\pi},\pi_s) ds+g(X_T^{t,x,\pi}) \right].$$
Setting $\Sigma \coloneqq [\sigma,\nu]$,
$H_{t}^{\text{cv}}(x,p,Z,u)\coloneqq b_{t}(x,u)  p+ \frac 1 2 \operatorname{tr}( (\Sigma \Sigma^{\top})_{t}(x,u) Z) + l_{t}(x,u)$ and $H_{t}(x,p,Z)\coloneqq\inf _{\pi \in U}H^{\text{cv}}_{t}(x,p,Z,\pi),$ 
the HJB equation is
\begin{equation}\label{eq:HJB_determinsitic}
    \partial_t v+H_{t}(x,D_{x} v,D^2_{x} v)=0, \quad  \forall(t, x) \in(0, T) \times \mathbb R^d, \quad v(T, x) =g(x), \quad  \forall x \in \mathbb R^d .
\end{equation}
We will assume that the value function $V \in C^{1,2}((0, T) \times \mathbb{R}^d) \cap C((0, T] \times \mathbb{R}^d)$ so that it is a classical solution of \eqref{eq:HJB_determinsitic}.  For  $h \in C^{1,2}( (0,T)\times \mathbb R^d)$, we will denote
\begin{small}
\begin{equation*}
\begin{aligned}
M_{t,T}^{t,x,\pi,h} &:=h_{T}(X_T^{t,x,\pi})-h_{t}(x)-\int_t^T \partial_{s} h_{s}(X^{t, x, \pi}_{s})+D h_{s}(X_s^{t,x,\pi})b_{s}(X_s^{t,x,\pi},\pi_s) \\
&\quad + \frac 1 2 \operatorname{tr} \left( (\Sigma\Sigma^{\top})_{s}(X_s^{t,x,\pi},\pi_s) D^{2} h_{s}(X_s^{t,x,\pi}) \right) ds\\
&=\int_t^T D h_{s}(X_s^{t,x,\pi}) \sigma_{s}(X_s^{t,x,\pi},\pi_s) dB_s+\int_t^T D h_{s}(X_s^{t,x,\pi})  \gamma_{s}(X_s^{t,x,\pi},\pi_s) dW_s,
\end{aligned}
\end{equation*}
\end{small}
where we have used  It\^{o}'s formula. In this setting, we have the following result, whose proof is the same as the one of \cref{th:duality_stoch_viscosity}. The theorem can also be stated in the non-smooth setting as in  \cite{bank_duality_2026}.
\begin{theorem}\label{th:duality_C12}Let $V \in C^{1,2}((0, T) \times \mathbb{R}^d) \cap C((0, T] \times \mathbb{R}^d)$, then almost surely
\begin{small}
\begin{align*}
V_{t}(x)&=\max_{h \in \mathcal H} V^{1, h}_{t}(x)=\max_{h \in \mathcal H}  V^{2, h}_{t}(x),
\end{align*}
\end{small}
where $\mathcal H:=\{h \in C^{1,2}( (0,T)\times \mathbb R^d): h_{T}(x)=g(x)  \}$ and
\begin{align*}
& V^{1, h}_{t}(x):=\mathbb E\Big[\essinf_{\pi \in {\mathcal U}} \mathbb E\left[ \int_t^T l_{s}(X_s^{t,x,\pi},\pi_s) ds+g(X_T^{t,x,\pi})  - M_{t,T}^{t,x,\pi,h} \bigg| \mathfrak F_T^W\right]  \Big]
,\\&\quad \quad \quad \quad \equiv  h_{t}(x) + \mathbb E\Bigg[\essinf_{\pi \in {\mathcal U}}\mathbb E\bigg[ \int_t^T\partial_s h_{s}(X_s^{t,x,\pi})+H^{\text{cv}}_{s}(X_s^{t,x,\pi},D h_{s}(X_s^{t,x,\pi}) ,D^{2} h_{s}(X_s^{t,x,\pi}),\pi_s)    ds \bigg| \mathfrak F_T^W\bigg] \Bigg],\\
&V^{2, h}_{t}(x):=h_{t}(x) + \int_t^T \inf_{y}  \left[  \partial_s h_{s}(y) +  H_{s}(y,D h_{s}(y),D^{2}h_{s}(y)) \right] ds.
\end{align*}
\end{theorem}
We note that the inner optimization problem of $V^{1,h}$ can under suitable assumptions be considered as a rough stochastic control problem  following \cref{section_rough_HJB}, which we spell out for the convenience of the reader.
\begin{coroll}
\label{cor:Markovian_rough_val}
Suppose that 
    $(b, \sigma, \gamma, l): [0,T]\times \mathbb{R}^{d}\times U \to \mathbb{R}^{d} \times \mathcal{L}(\mathbb{R}^{m_{B}}; \mathbb{R}^{d})\times \mathcal{L}(\mathbb{R}^{m_{W}}; \mathbb{R}^{d})$
are continuous such that 
\begin{enumerate}
    \item $b,\sigma, l$ are globally bounded and Lipschitz in $\mathbb{R}^{d}$, uniformly over $[0,T]\times U$.
    \item $\gamma$ is independent of $\pi$. 
    \item There is $\beta> \nicefrac{1}{\alpha}$ and  $\gamma'$ such that $(\gamma, \gamma')\in \mathscr{D}^{2\alpha}_{Z}C_{b}^{\beta}, (D\gamma, D\gamma')\in \mathscr{D}^{2\alpha}_{Z}C_{b}^{\beta-1}$ for any $Z \in \mathcal{C}^{\alpha}([0,T]; \mathbb{R}^{m_{W}})$. 
\end{enumerate}
Then for any $\mathbf{Z}\in \mathscr{C}^{\alpha}([0,T]; \mathbb{R}^{m_{W}}), q \in[2, \infty)$ and control $\pi \in \mathcal{U}$ the controlled RSDE 
\begin{equation*}
\begin{aligned}
    dX_{s}^{t, x, \pi; \mathbf{Z}}= b_{s} (X_{s}^{t, x, \pi; \mathbf{Z}}, \pi_{s}) ds + \sigma_{s} (X_{s}^{t, x, \pi; \mathbf{Z}}, \pi_{s}) dB_{s}+ \gamma_{s}(X_{s}^{t, x, \pi; \mathbf{Z}}) d\mathbf{Z}_{s}; \quad X^{t, x, \pi; \mathbf{Z}}_{t}=x 
\end{aligned}
\end{equation*}
has a unique $L_{q, \infty}$-solution $X^{t,x,\pi, \mathbf{Z}}$ and further defining the rough stochastic control problem for any $h \in \mathcal{H}$
\begin{equation}
\label{eq:Markov_rough_stochastic_val}
        \mathcal{V}_{t}(x; \mathbf{Z})\coloneqq \inf_{\pi \in \mathcal{A}} \mathbb{E}'\left[ \int_{t}^{T} \partial_s h_{s}(X_s^{t,x,\pi; \mathbf{Z}})+H^{\text{cv}}_{s}(X_s^{t,x,\pi; \mathbf{Z}},D h_{s}(X_s^{t,x,\pi; \mathbf{Z}}) ,D^2 h_{s}(X_s^{t,x,\pi; \mathbf{Z}}),\pi_s)  ds \right]
\end{equation}
it holds for any $t\in [0,T]$ and $x \in \mathbb{R}^{d}$ almost surely 
\begin{equation*}
    \mathcal{V}_{t}(x; \mathbf{W}^{\text{It\^{o}}})=\esssup_{\pi \in {\mathcal U}}\mathbb E\bigg[ \int_t^T\partial_s h_{s}(X_s^{t,x,\pi})+H^{\text{cv}}_{s}(X_s^{t,x,\pi},D h_{s}(X_s^{t,x,\pi}) ,D^{2} h_{s}(X_s^{t,x,\pi}),\pi_s)    ds \bigg| \mathfrak F_T^W\bigg] 
\end{equation*}
\end{coroll}
\begin{proof}
    This is a special case of \cref{section_rough_HJB} and in particular \cref{thm:randomization_value_function} for deterministic coefficients.
\end{proof}
\begin{remark}
We emphasize, that \cref{cor:Markovian_rough_val} generalizes \cite[Theorem 16]{DiehlFrizGassiat2017} in two ways:
\begin{enumerate}
    \item It disentangles the use of rough analysis from classical stochastic analysis. Notably \cref{th:duality_C12}, is a purely probabilistic claim, making no use of rough paths at all. However as 
    shown throughout \cref{section_rough_HJB} rough paths provide a suitable toolbox to tackle the inner control problem. 
    \item It allows to consider (partially) controlled diffusion coefficients, which is not possible in the setting of \cite{DiehlFrizGassiat2017} due to the reasons mentioned in \cref{remark:controlled_diffusion}.
\end{enumerate}
The latter however comes at the cost of considering a 2nd-order nonlinearity in the rough HJB equation corresponding to \eqref{eq:Markov_rough_stochastic_val} see \cref{subsection_Viscosity Solutions to the Rough Stochastic Control Problems}. 
\end{remark}
\section{A compendium on measure theory}\label{measure_theory}
Here we collect results from measure theory used throughout the work. 
\subsection{Regular conditional distributions and essential infima}
Suppose that $\Omega$ is Polish and denote by  $(\Omega, \mathfrak{F}, (\mathfrak{F}_{t})_{t \in [0,T]}, \mathbb{P})$ a filtered probability space on $\Omega$ with $\mathfrak{F}=\mathfrak{B}(\Omega)$.
\begin{lemma}
\label{lemma:cond_distr_nullsets}
Let $t \in [0,T]$ and $\mathbb{P}_{t}$ be a version of the regular conditional distribution $\mathbb{P}(\cdot| 
\mathfrak{F}_{t})$\footnote{which exists as $\Omega$ is Polish and equipped with $\mathfrak{B}(\Omega)$}. Then for any $A \in \mathfrak{F}$, it holds $\mathbb{P}_{t}(A)=0$ $\mathbb{P}$-a.s. iff $\mathbb{P}(A)=0$. 
\end{lemma}
\begin{proof}
    Suppose $\mathbb{P}_{t}(A)=0$ $\mathbb{P}$-a.s. Then it holds $\mathbb{P}(A)=\mathbb{E}[\mathbb{P}_{t}(A)]=0$. Conversely if $0=\mathbb{P}(A)=\mathbb{E}[\mathbb{P}_{t}(A)]$, then using that $\mathbb{P}_{t}(A)\geq 0$ $\mathbb{P}$-a.s. it follows $\mathbb{P}_{t}(A)=0$ $\mathbb{P}$-a.s.
\end{proof}
Importantly, we conclude that essential suprema/infimia w.r.t. $\mathbb{P}$ and $\mathbb{P}_{t}$ align.
\begin{coroll}
\label{cor:conditional_essinf}
    Let $A$ be a set of parameters and consider $X: \Omega \times A \to \mathbb{R}$ such that for any $a\in A$ $X(\cdot, a)$ is $\mathfrak{F}/\mathfrak{B}$-measurable. Let $\mathbb{P}_{t}$ be 
    as in \cref{lemma:cond_distr_nullsets}. Then the essential infimum w.r.t. the conditional distribution $\mathbb{P}_{t}$ is defined as a random variable $Y\in L^{0}(\Omega,\mathfrak{F}; \mathbb{R})$ such that
    \begin{enumerate}
        \item \label{item1_essinf} $\mathbb{P}_{t}(Y \leq X(a)) =1$ $\mathbb{P}$-a.s., for any $a \in A$
        \item \label{item2_essinf} $\mathbb{P}_{t}({Y \geq \tilde Y})=1$ $\mathbb{P}$-a.s., for any random variable $\tilde{Y}\in L^{0}(\Omega,\mathfrak{F}; \mathbb{R})$ satisfying \eqref{item1_essinf}. 
    \end{enumerate}
     Such a $Y$ exists and it holds $\mathbb{P}$-almost-surely
    \begin{equation*}
        \essinf_{t; a \in A}X(\cdot, a) \coloneqq Y=  \essinf_{a \in A}X(\cdot, a).
    \end{equation*}
\end{coroll}
\begin{proof}
    We show the claim by verifying that $Y \coloneqq \essinf_{a \in A}X(\cdot, a)$ (with the $\essinf$ taken w.r.t. $\mathbb{P}$)  satisfies \eqref{item1_essinf} and \eqref{item2_essinf}. First note that by definition of the essential infimum w.r.t. $\mathbb{P}$ it holds $\mathbb{P}(Y \leq X(a)) =1$ and thus by  \cref{lemma:cond_distr_nullsets}, \eqref{item1_essinf} follows. Now suppose we have another $\tilde{Y}\in L^{0}(\Omega)$ such that \eqref{item1_essinf} is satisfied. Then  also by \cref{lemma:cond_distr_nullsets} it holds $\mathbb{P}(\tilde{Y} \leq X(a))=1$ for any $a \in A$ and thus again by definition of the essential infimum it holds $\mathbb{P}(Y \geq \tilde{Y})=1$ and so again by \cref{lemma:cond_distr_nullsets}, \eqref{item2_essinf} follows. 
\end{proof} 
\subsection{Predictability, optionality and non-anticipative functionals on path space}
\begin{definition}[\cite{stricker_calcul_1978}]
Consider a filtered probability space $(\Omega, \mathfrak{F}, (\mathfrak{F}_{t})_{t \in [0,T]}, \mathbb{P})$ and Polish spaces $A, V$ with Borel-algebras $\mathfrak{B}_{A}, \mathfrak{B}_{V}$. A parametrized, $V$-valued, measurable stochastic process $X: [0,T]\times \Omega \times A\to V$ is said to be $\mathfrak{B}_{A}$\textit{- predictable} if $X$ is $\mathfrak{P}\otimes \mathfrak{B}_{A}/ \mathfrak{B}_{V}$-measurable, where $\mathfrak{P}$ denotes the predictable $\sigma$-algebra on $[0,T]\times \Omega$ w.r.t. $(\mathfrak{F}_{t})$.
\end{definition}
In the following let $(\Omega, \mathfrak{F}, (\mathfrak{F}_{t})_{t \in [0,T]}, \mathbb{P})$ the Wiener space with augmented Brownian filtration and $V$ some Polish space. 
\begin{lemma}[Lemma 5.5 in \cite{flz26}]
\label{appendix:lemma:optional_nonantic}
It holds
\begin{enumerate}
\item $(\mathfrak{F}_{t})$-predictable processes are indistinguishable from $(\mathfrak{F}_{t})$-optional ones. 
\item Let $X: [0,T]\times \Omega \to V$ measurable. Then $X$ is $(\mathfrak{F}_{t})$-optional w.r.t.  iff  $X_{t}(\omega) = X_{t}(\omega_{\cdot \wedge t})$ up to indistinguishability.
\end{enumerate}
\end{lemma}
\subsection{Products of probability spaces}
\label{appendix:subsection:product_prob_spaces}
    Let $(\Omega' , \mathfrak{F}', \mathbb{P}')$ and $(\Omega'' , \mathfrak{F}'', \mathbb{P}''), $ be two probability spaces. Then we define the product of probability spaces by 
    \begin{equation*}
        (\Omega', \mathfrak{F}', \mathbb{P}')\otimes (\Omega'', \mathfrak{F}'', \mathbb{P}'')\coloneqq (\Omega' \times \Omega '', \mathfrak{F}' \otimes \mathfrak{F}'', 
        \mathbb{P}' \otimes \mathbb{P}''),
    \end{equation*}
    where $\mathfrak{F}' \otimes \mathfrak{F}''$ is the product-$\sigma$-algebra and $\mathbb{P}'\otimes \mathbb{P}''$ is the product-measure. 
In \cref{section_rough_HJB} we work on the product space defined like this: 
For brevity denote by $C^{d}\coloneqq C([0,T]; \mathbb{R}^{d})$. We consider the canonical path-space $(\Omega', \mathfrak{F}', \mathbb{P}')\coloneqq (C^{m_{B}}, \mathfrak{B}(C^{m_{B}}), \mathbb{P}')$, where $\mathbb{P}'$ is the Wiener measure w.r.t. the canonical process $B: [0,T]\times \Omega' \to \mathbb{R}^{m_{B}}; (t, \omega') \mapsto \omega'_{t}$ i.e. such that $B$ is a $m_{B}$-dim. Brownian Motion under $\mathbb{P}^{1}$. Further $(\Omega'', \mathfrak{F}'', \mathbb{P}'')\coloneqq (C^{m_{W}}, \mathfrak{B}(C^{m_{W}}), \mathbb{P}^{2})$ where $\mathbb{P}^{2}$ is the Wiener measure w.r.t. the canonical process $W$ on $\Omega''$. 
\section{A compendium on RSDEs and good rough paths}\label{appendix_section_rsde_and_control}
\subsection{Rough stochastic analysis and differential equations}
\label{appendix_section_rsde_control}
In the following we give a minimal exposition to \textit{Rough stochastic analysis} and \textit{rough stochastic differential equations}(RSDEs) based on \cite{fhl21} for the convenience of the reader. 
\newline
Throughout the following let $2 \leq q \leq r \leq \infty$, $q < \infty$, $\alpha \in (\nicefrac{1}{3}, \nicefrac{1}{2}]$ and $\gamma \in (1, \infty)$ be fixed. Further let $V, W$ be finite-dim. Banach spaces and $\mathbf{X}\in \mathscr{C}^{\alpha}([0,T]; V)$ fixed.\\
For any probability space $(\Omega, \mathfrak{F}, \mathbb{P})$ with sub-$\sigma$-algebra $\mathfrak{G}\subset \mathfrak{F}$ we define
\begin{equation*}
    \Vert \Vert \cdot |\mathfrak{G} \Vert_{p} \Vert_{q}\coloneqq \left \Vert \mathbb{E}[|\cdot |^{p}|\mathfrak{G}]^{\frac{1}{p}} \right \Vert_{q}. 
\end{equation*}
Given a filtered probability space $(\Omega, \mathfrak{F}, (\mathfrak{F}_{t})_{t\in [0,T]}, \mathbb{P})$, for any strongly $\mathfrak{B}(\Delta_{T})\otimes \mathfrak{F}/\mathfrak{B}(V)$-measurable $A: \Delta_{T} \times \Omega \to V$ we say $A \in C_{2}^{\kappa}L_{q, r}([0,T]; V)$ for $\kappa >0$ if 
\begin{equation*}
    \Vert A \Vert_{\kappa; q, p}\coloneqq \sup_{(s,t) \in \Delta_{T}} \frac{\Vert \Vert A_{s,t} |\mathfrak{F}_{s}\Vert_{p}\Vert_{q}}{|t-s|^{\kappa}}< \infty.
\end{equation*}
\begin{definition}[\cite{fhl21} Definition 4.1]
\label{def:random_bounded_Lipschitz}
Let $g: [0,T]\times \Omega \to C_{b}(\mathbb{R}^{d}; \mathbb{R}^{m})$ strongly $\mathfrak{B}([0,T])\otimes \mathfrak{F}/ \mathfrak{C}_{b}(\mathbb{R}^{d}; \mathbb{R}^{m})$-measurable. We say that $g$ is \textit{random bounded Lipschitz} if both 
\begin{equation*}
    \begin{aligned}
        \sup_{t \in [0,T]} \esssup_{\omega \in \Omega} \sup_{x \in \mathbb{R}^{d}} |g_{t}(\omega, x)|< \infty, \quad \sup_{t \in [0,T]} \esssup_{\omega \in \Omega} \sup_{x,y \in \mathbb{R}^{d}} \frac{|g_{t}(\omega, x)-g_{t}(\omega, y)|}{|x-y|}< \infty.
    \end{aligned}
\end{equation*}
    
\end{definition}
	\begin{definition}[\cite{fhl21} Definition 3.1]
	\label{def:SCRP}
		We say that $(Z,Z')$ is a stochastic $X$-controlled rough path of $(q, r)$-integrability and $\alpha$-H\"older regularity with values in $W$ with respect to $(\mathfrak{F}_{t})$
		if the following are satisfied
		\begin{enumerate}
		 	\item $Z\colon \Omega\times  [0,T]\to W$ and $Z'\colon[0,T]\times \Omega\to \mathcal L(V;W)$
		 	are $(\mathfrak{F}_{t})$-progressively measurable;
            \item 
			$\delta Z$ belongs to $C^{\alpha}_2 L_{q,r}( [0,T],W;\mathcal{L}(V;W))$, $\delta Z'$ belongs to $C^{\alpha} L_{q,r}([0,T]; \mathcal{L}(V; W))$
			\\and \(\sup_{t\in [0,T]}\|Z'_t\|_{r}<\infty\)
		 	\item Defining
		 	$R^Z_{s,t}
		 	\coloneqq \delta Z_{s,t}-Z'_s \delta X_{s,t}
		 	,\quad\text{for}\enskip  (s,t)\in \Delta_{T}$
		 	we have that $\E_{\bullet} R^Z \coloneqq (\mathbb{E}_{s}[R^{Z}_{s,t}])_{(s,t)\in \Delta_{T}}$ belongs to $C^{2\alpha}_2L_{r}([0,T];W) $;
            
		\end{enumerate}
		The class of such processes is denoted by $\mathbf{D}_X^{2\alpha}L_{q,r}([0,T
        ];W)$, or simply
	 $\mathbf{D}_X^{2\alpha}L_{q,r}$ whenever clear from the context.
	\end{definition}
	\begin{definition}[\cite{fhl21} Definition 3.7]\label{def:scvec}
		We call a tuple $(f,f'): [0,T]  \to C^\gamma_{b}(W, \bar W) \times C^{\gamma-1}_{b} (W,\mathcal L(V,\bar W))$ a {\em controlled vector field on \( W \)}
		of $\alpha$-regularity with respect to $X$
		if the following conditions are satisfied.
		\begin{enumerate}

			\item
			Letting $\bk{Z}_{\kappa}:=\sup_{(s,t)\in \Delta_{T}}\frac{|\sup_{x \in W} |Z_{s,t}(x)|}{|t-s|^{\kappa}}$, 
			the quantities
			$\bk{\delta f}_{\alpha}$, $\bk{\delta f'}_{\alpha}$, \( \bk{\delta Df}_{\alpha} \) are finite.%

			\item Defining  
            $R^{f}_{s,t}(x)\coloneqq f_{t}(x)-f_{s}(x)- f'_{s}(x) \delta X_{s,t}$ for any $(s, t) \in \Delta_{T}$ and $x \in W$ it holds $\bk{R^{f}}_{2\alpha}< \infty$
		\end{enumerate}
		The class of such vector fields is denoted by $\mathscr{D}^{2\alpha}_X\mathcal{C}^\gamma_b( [0,T];W; \bar{W})$, or simply $\mathscr{D}_X^{2\alpha}\mathcal{C}^\gamma_b$ whenever the domains $( [0,T];W; \bar{W})$ are clear from context. For two paths $X, \tilde{X}\in \mathcal{C}^{\alpha}([0,T];V)$ and $(f, f')\in \mathscr{D}^{2\alpha}_X\mathcal{C}^\gamma_b( [0,T];W; \bar{W}), (\tilde{f}, \tilde{f}') \in \mathscr{D}^{2\alpha}_{\tilde{X}}\mathcal{C}^\gamma_b( [0,T];W; \bar{W})$ we define 
        \begin{equation*}
            [ f, f'; \tilde{f}, \tilde{f}']_{X, \tilde{X}; 2\alpha}\coloneqq \llbracket \delta f- \delta \tilde{f} \rrbracket_{\alpha}+ \llbracket \delta f'- \delta \tilde{f}' \rrbracket_{\alpha}+\llbracket \delta Df- \delta D\tilde{f} \rrbracket_{\alpha}+ \llbracket \delta R^f- \delta R^{\tilde{f}}\rrbracket_{2\alpha}
        \end{equation*}
	\end{definition}
    We are interested in giving well-posed meaning to the following \textit{rough stochastic differential equation} (RSDE)
    \begin{equation}\label{eqn.srde}
		dY_t(\omega)=b_t(\omega,Y_t(\omega))dt+\sigma_t(\omega,Y_t(\omega))dB_t(\omega)+(f_t,f'_t)(Y_t(\omega)) d\mathbf{X}_t,\quad t\in[0,T].
	\end{equation}
    \begin{definition}[\cite{fhl21} Definition 4.2]
		\label{def.soln}
        Let $\alpha \in (\nicefrac{1}{3}, \nicefrac{1}{2}]$, $2\leq q \leq r \leq \infty, q< \infty$.
		An $L_{q,r}$-integrable solution of \eqref{eqn.srde}  over $[0,T]$ is a continuous $(\mathfrak{F}_{t})$-adapted process $Y$ such that the following conditions are satisfied
		\begin{enumerate}
        \item $\int_0^T |b_r(Y_r)|dr$ and $\int_0^T |(\sigma \sigma^\top)_r(Y_r)|dr$ are finite a.s.;
			\item $(f(Y),Df(Y)f(Y)+f'(Y))$ belongs to $\mathbf{D}_X^{2\alpha}L_{q,r}([0,T];\mathcal{L}(V;W))$
			\item\label{def.Davie.expansion} $Y$ satisfies the following stochastic Davie-type expansion
			\begin{equation*} \|\|Y^{\sharp}_{s,t}|\mathfrak{F}_{s}\|_q\|_r=o(t-s)^{1/2}; \quad 
				\|\mathbb{E}_s(Y^{\sharp}_{s,t})\|_r=o(t-s)
			\end{equation*}
			for every $(s,t)\in \Delta_{T}$, where
			\begin{equation*}
            \begin{aligned}
				Y^{\sharp}_{s,t}
				\coloneqq&\delta Y_{s,t}-\int_s^tb_r(Y_r)dr-\int_s^t \sigma_r(Y_r)dB_r -f_s(Y_s)\delta X_{s,t}-\left(Df_s(Y_s)f_s(Y_s)+f'_s(Y_s)\right)\mathbb{X}_{s,t}.
        \end{aligned}
        \end{equation*}
		\end{enumerate}
		When the initial datum $Y_0=\xi$ is specified, we say that $Y$ is a solution starting from $\xi$.
	\end{definition}
    The following is a simplified case included in the main well-posedness theorem of such solutions \cite[Theorem 4.6]{fhl21}, which is sufficient for this work.
	\begin{theorem}
    \label{thm.fixpoint}
			 Let	$b,\sigma$ be random bounded Lipschitz functions,
		 assume that $(f,f')$ belongs to  $\mathscr{D}^{2\alpha}_{X}\mathcal{C}^\gamma_b$ while $(Df,Df')$ belongs to  $\mathscr{D}^{2\alpha}_{X}\mathcal{C}^{\gamma-1}_b$.
			Assume moreover that $\alpha\in (\nicefrac{1}{3}, \nicefrac{1}{2}]$ and $\gamma> \nicefrac{1}{\alpha}$. Then for every $\xi\in L^{0}(\mathfrak{F}_0;W)$ and $q \in [2, \infty)$, there exists a unique $L_{q,\infty}$-integrable  solution on $[0,T]$ to \eqref{eqn.srde} starting from $\xi$.
		\end{theorem}
\subsection{Good rough paths}
\label{subsection:good_rough_paths}
We adapt the notion of 'good rough paths' from \cite{coutin_2007} to our setting.
\begin{definition}
\label{def:good_rough_path}
Let $\mathbf{X}=(X, \mathbb{X})\in \mathscr{C}^{0, \alpha}_{g}([0,T]; V)$ and $(X^{n})_{n \in \mathbb{N}}$ a sequence of continuous $V$-valued paths of bounded variation converging to $X$ in $\Vert \cdot \Vert_{\alpha}$. Then $(X^{n})_{n \in \mathbb{N}}$ is a \textit{good rough path} sequence (associated to $\mathbf{X}$) if for 
\begin{equation*}
\begin{aligned}
    (\mathbf{X}; X^{n})&\coloneqq \left(\left(
    \begin{array}{c}
    X\\
    X^{n}
    \end{array} \right),\left( \begin{array}{cc}
    \mathbb{X} & \int \delta X \otimes dX^{n}\\
    \int \delta X^{n} \otimes dX & \int \delta X^{n} \otimes dX^{n}
    \end{array}\right)\right)\eqqcolon \left(\left(
    \begin{array}{c}
    X\\
    X^{n}
    \end{array} \right),\left( \begin{array}{cc}
    \mathbb{X} & \Pi(X; X^{n})\\
    \Pi(X^{n}; X) & \Pi(X^{n}; X^{n})
    \end{array}\right)\right);\\ 
    (\mathbf{X}; \mathbf{X})&\coloneqq \left( \left( \begin{array}{c} X\\ X\end{array} \right), \left(\begin{array}{cc} \mathbb{X} & \mathbb{X} \\ \mathbb{X} & \mathbb{X} \end{array} \right) \right)
\end{aligned}
\end{equation*}
it holds $
    \rho_{\alpha}((\mathbf{X}; X^{n}), (\mathbf{X}; \mathbf{X}))\stackrel{n \to \infty}{\longrightarrow}0$.
\end{definition}
Note, that not all geometric rough paths are 'good' (in particular no 'pure area rough path' can be good; see \cite[Example 3]{coutin_2007}). However importantly for our application piece-wise linear approximations of Brownian Motion are good.
\begin{example}[Theorem 9 in \cite{coutin_2007}]
\label{appendix:ex:good_rp_Brownian}
Let $W$ be a $m_{W}$-dim. Brownian Motion. Then there is $\mathcal{P}^{n}=(t_{i}^{n})_{i=0}^{n}$ a sequence of subdivisions (partitions where $\mathcal{P}^{n+1}\subset \mathcal{P}^{n}$ for any $n \in \mathbb{N}$) of $[0,T]$ of locally vanishing mesh size such that the linear interpolation
\begin{equation*}
    W^{n}_{t}\coloneqq W_{t_{i}^{n}}+ \frac{t-t_{i}^{n}}{t_{i+1}^{n}-t_{i}^{n}} \delta W_{t_{i}, t_{t_{i+1}}}; \quad t\in [t_{i}^{n}, t_{i+1}^{n})
\end{equation*}
is a.s. a good rough path sequence w.r.t. $\mathbf{W}^{\text{Strat}}$.
\end{example}
The following is an extension of \cite[Corollary 3.5]{fhl21}.
\begin{lemma}
\label{lemma:good_rp_approx}
    Let $\alpha \in (\nicefrac{1}{3}, \nicefrac{1}{2}]$, $\nicefrac{1}{\alpha}< q \leq r \leq \infty$ and $q< \infty$. Let $\mathbf{X}\in \mathscr{C}^{0, 
    \alpha}_{g}([0,T]; V)$, $(Y, Y'), (\bar{Y}, \bar{Y}')\in \mathbf{D}^{2\alpha}_{X}L_{q,r}([0,T]; \mathcal{L}(V; W))$ and $(X^{n})_{n \in \mathbb{N}}\subset \mathcal{C}^{1}([0,T]; V)$. Then for the integrals\footnote{By $\Vert \cdot \Vert_{q}\leq \Vert \cdot \Vert_{q, r}$ and $\nicefrac{1}{\alpha}< q$ one infers by the Kolmogorov continuity criterion that $\bar{Y}$ is continuous almost surely. Thus $\bar{Z}^{n}$ is the classical Riemann-Stieltjes integral. Recall that 
   \cref{def:SCRP} makes a priori no claims on the regularity of the sample paths itself.} and remainders
    \begin{equation*}
        \begin{aligned}
            Z_{t}&\coloneqq \int_{0}^{T} (Y, Y')_{r} d\mathbf{X}_{r}; \quad \bar{Z}^{n}_{t} \coloneqq \int_{0}^{T} \bar{Y}_{r} dX^{n}_{r}; \\
            R^{Z}_{s,t}&\coloneqq \delta Z_{s,t}- Y_{s} \delta X_{s,t} - Y'_{s} \mathbb{X}_{s,t}; \quad R^{Z^{n}}_{s,t} \coloneqq \delta Z^{n}_{s,t} - \bar{Y}_{s} \delta X^{n}_{s,t} - \bar{Y}'_{s} \Pi(X; X^{n})_{s,t}\\
            R^{Y}_{s,t}&\coloneqq \delta Y_{s,t}- Y'_{s}\delta X_{s,t}; \quad R^{\bar{Y}}_{s,t}\coloneqq \delta \bar{Y}_{s,t}- \bar{Y}'_{s} \delta X_{s,t}
        \end{aligned}
    \end{equation*}
    it holds
    \begin{equation*}
    \begin{aligned}
    \left \Vert Z- \bar{Z}^{n} \right \Vert_{\alpha; q, r}&+ \left \Vert \mathbb{E}_{\bullet}(R^{Z})-\mathbb{E}_{\bullet}(R^{\bar{Z}^{n}})\right \Vert_{3\alpha; r}+\left \Vert \left \Vert \sup_{t \in [0,T]}|Z_{t}-\bar{Z}^{n}_{t}|\bigg| \mathfrak{F}_{0}\right \Vert_{q} \right \Vert_{r}\\
    &\lesssim \rho_{\alpha}(\mathbf{X}; (X^{n}, \Pi(X; X^{n})))+ \Vert Y, Y'; \bar{Y}, \bar{Y}' \Vert_{X, X; 2\alpha;  q, r}, 
    \end{aligned}
    \end{equation*}
where 
\begin{equation*}
    \Vert Y, Y'; \bar{Y}, \bar{Y}' \Vert_{X, X; 2\alpha;  q, r}\coloneqq \sup_{t \in [0,T]} \Vert Y'_{t}- \bar{Y}'_{t}\Vert_{r}+ \Vert \delta Y- \delta Y'\Vert_{\alpha; q, r}+\Vert \delta Y'- \delta \bar{Y}'\Vert_{\alpha; q, r}+ \Vert \mathbb{E}_{\bullet}[R^{Y}]- \mathbb{E}_{\bullet}[R^{Y'}]\Vert_{2\alpha; r}
\end{equation*} and the implicit constant only depends on $\Vert \bar{Y}, \bar{Y}' \Vert_{X; \alpha; q, r}$ and $\Vert \mathbf{X}\Vert_{\alpha}$.
\end{lemma}
\begin{proof}
We define 
\begin{equation*}
\begin{aligned}
\Xi^{n}_{s,t} &\coloneqq Y_{s} \delta X_{s,t} + Y'_{s} \mathbb{X}_{s,t} - \bar{Y}_{s} \delta X^{n}_{s,t} - \bar{Y}'_{s} \Pi(X; X^{n}_{s,t})\\
&=(Y_{s}- \bar{Y}_{s}) \delta X_{s,t}+ (Y'_{s}- \bar{Y}_{s}')\mathbb{X}_{s,t}+ \bar{Y}_{s}(\delta X_{s,t}- \delta X^{n}_{s,t}) + \bar{Y}_{s}'(\mathbb{X}_{s,t}- \Pi(X; X^n)_{s,t})
\end{aligned}
\end{equation*}
    For any $(s,u,t) \in \Delta_{T}^{(3)}$ it holds $\delta \Pi (X; X^{n})_{s,u,t}= \delta X_{s,u}\otimes \delta X^{n}_{u,t}$
    and thus
    \begin{equation*}
    \begin{aligned}
        \delta \Xi^{n}_{s,u,t}&= - (\delta(Y-\bar{Y})_{s,u}- (Y'-\bar{Y}')_{s}\delta X_{s,u}) \delta X_{u,t}- \delta (Y'- \bar{Y}')_{s,u}\mathbb{X}_{s,u}\\
        &-(\delta \bar{Y}_{s,u}- \bar{Y}_{s}' \delta X_{s,u} ) (\delta X_{u,t}^{n}- \delta X_{u,t})- \delta \bar{Y}'_{s,u} (\Pi(X; X^{n})_{u,t}- \mathbb{X}_{u,t}).
    \end{aligned}
    \end{equation*}
    One easily checks that 
    \begin{eqnarray*}
        \frac{\left \Vert \mathbb{E}_{s} (\delta \Xi_{s,u,t}^{n}) \right \Vert_{r}}{|t-s|^{3\alpha}}&\leq& \Vert \mathbb{E}_{\bullet} (R^{Y})- \mathbb{E}_{\bullet}(R^{\bar{Y}}) \Vert_{2\alpha; r} \Vert X \Vert_{\alpha} + \Vert Y'- \bar{Y}' \Vert_{\alpha; q, r}\Vert \mathbb{X} \Vert_{2\alpha}\\
        &&+\Vert \mathbb{E}_{\bullet}(R^{\bar{Y}})\Vert_{2\alpha; r} \Vert X - X^{n}\Vert_{\alpha}+\Vert \bar{Y}'\Vert_{\alpha; q,r} \Vert \Pi(X; X^{n})-\mathbb{X}\Vert_{2\alpha}; \\
        \sup_{u \in [0,T]}\frac{\left \Vert \delta \Xi^{n}_{s,u,t} \big| \mathfrak{F}_{s} \right \Vert_{q,r}}{|t-s|^{2\alpha}}&\leq & \frac{ \left \Vert \left \Vert  \sup_{r \in [\frac{s+t}{2}, t]} |\delta \Xi^{n}_{s,\frac{s+t}{2}, r}| \mathfrak{F}_{s} \right \Vert_{q} \right \Vert_{r}}{|t-s|^{2\alpha}} \\
        &\leq &\left( \Vert Y- \bar{Y} \Vert_{\alpha; q,r} + \sup_{r \in [0,T]} \Vert Y'_{r} - \bar{Y}'_{r} \Vert_{r} \Vert X \Vert_{\alpha}\right) \Vert X \Vert_{\alpha} \\
        &&+\left( \Vert \bar{Y} \Vert_{\alpha; q,r}+ \sup_{r \in [0,T]} \Vert \bar{Y}_{r}'\Vert_{r} \Vert X\Vert_{\alpha}\right)\Vert X^{n}- X \Vert_{\alpha}\\
        &&+\sup_{r \in [0,T]} \Vert Y'_{r}- \bar{Y}'_{r} \Vert_{r} \Vert \mathbb{X} \Vert_{2\alpha}+ \sup_{r \in [0,T]} \Vert \bar{Y}'_{r} \Vert_{r} \Vert \Pi(X; X^{n}) - \mathbb{X} \Vert_{2\alpha}.
    \end{eqnarray*}
    Thus by applying the mixed-moment stochastic sewing lemma \cite[Theorem 2.9]{fhl21} to $\Xi^{n}$ the claim follows.
\end{proof}
\begin{remark}
    We note that for $(Y, Y')\equiv (\bar{Y}, \bar{Y}')$, \cref{lemma:good_rp_approx} asserts that rough stochastic integration for stochastic controlled rough paths is stable under good rough path approximation. This is an extension of \cite[Lemma 7]{coutin_2007} to (stochastic) controlled integrands and seems to be novel in the setting of controlled rough paths even for the deterministic case.
\end{remark}
Using \cref{lemma:good_rp_approx} we investigate the stability of RSDEs driven by good rough path sequences.
\begin{lemma}
\label{lemma:good_rp_RSDE}
    Let $\mathbf{X}\in \mathscr{C}^{0, \alpha}_{g}([0,T]; V)$, $\alpha \in (\nicefrac{1}{3}, \nicefrac{1}{2}], q \in (\nicefrac{1}{\alpha}, \infty)$ and $(X^{n})_{n \in \mathbb{N}}\subset \mathcal{C}^{1}
    ([0,T]; V)$. Further denote by $Y$ the $L_{q,\infty}$-solution to
    \begin{equation*}
     dY_{t} = b_{t}(Y_{t}) dt + \sigma_{t}(Y_{t}) dB_{t} + (f, f')_{t}(Y_{t}) d\mathbf{X}_{t}; \quad Y_{0}= \xi
    \end{equation*}
    and by $Y^{n}$ the solution to 
    \begin{equation*}
        dY^{n}_{t}= b_{t}(Y^{n}_{t})dt + \sigma_{t}(Y^{n}_{t}) dB_{t} + f_{t}(Y^{n}_{t}) dX^{n}_{t}; \quad Y^{n}_{0}= \xi 
    \end{equation*}
    for $\xi \in L^{0}(\mathfrak{F}_{0}, W)$, coefficients $b, \sigma$ random bounded Lipschitz, $(f, f')\in \mathscr{D}^{2\alpha}_{X}C_{b}^{\gamma}$ and $(Df, Df')\in \mathscr{D}^{2\alpha}_{X}C^{\gamma-1}_{X}$ for $\gamma > \nicefrac{1}{\alpha}$. Then 
    \begin{equation}
    \label{eq:good_approxi_RSDE}
    \Vert Y- Y^{n}\Vert_{\alpha; q, \infty}\lesssim \rho_{\alpha}(\mathbf{X}; (X^{n}, \Pi(X; X^{n}))),
    \end{equation}
    where the implicit coefficient only depends on $(f, f')$. Further if $(X^{n})_{n \in \mathbb{N}}$ is a good rough path sequence w.r.t. $\mathbf{X}$, then 
    \begin{equation*}
        \lim_{n \to \infty} \Vert Y - Y^{n} \Vert_{\alpha; q, \infty}=0
    \end{equation*}
\end{lemma}
\begin{proof}
Existence and uniqueness of $Y^{n}$ is clear as the $dX^{n}$-term is the usual Riemann-Stieltjes integral. The proof of \eqref{eq:good_approxi_RSDE} is mutatis mutandis the proof of \cite[Theorem 4.9]{fhl21}, where at every instance we replace the usage of \cite[Corollary 3.5]{fhl21} by \cref{lemma:good_rp_approx}.
\end{proof}
\section{Smooth semimartingales and the It\^{o}-Kunita-Wentzell formula}\label{kunita}
We introduce the set of smooth semimartingale fields.
\begin{definition}[Definition 4.1 in \cite{qiu2018}]\label{def:smooth-semimartingales}
For $v \in \mathcal{S}^2\left(C_b\left(\mathbb{R}^d\right)\right) \cap \mathcal{L}^2\left(C^2_b\left(\mathbb{R}^d\right)\right)$, we say $v \in \mathcal{C}_{\mathfrak{F}}^2$ if there exists $\left(\mathfrak{d}_t v, \mathfrak{d}_\omega v\right) \in \mathcal{L}^2\left(C_b\left(\mathbb{R}^d\right)\right) \times \mathcal{L}^2\left(C^1_b\left(\mathbb{R}^d; \mathbb{R}^{1\times m_{W}}\right)\right)$ such that, a.s., 
$$
v_r( x)=v_T(x)-\int_r^T \mathfrak{d}_s v_s(x) d s-\int_r^T \mathfrak{d}_\omega v_s( x) d W_s \quad \forall(r, x) \in[0, T] \times \mathbb{R}^d
$$
\end{definition}

The Doob-Meyer decomposition theorem implies the uniqueness of the pair ($\mathfrak{d}_t v, \mathfrak{d}_\omega v$) and thus makes sense of the linear operators $\mathfrak{d}_t, \mathfrak{d}_{\omega}$.
\begin{lemma}[It\^{o}-Kunita-Wentzell formula \cite{Kunita1997}]\label{lemma:It\^{o}-kunita}
 Let $v \in \mathcal{C}_{\mathfrak{F}}^2$ and let $X_s^{t,x,\pi}$ be a solution of \eqref{eq:controlled_doubly_sde}, for  $x \in \mathbb{R}^d$, $\pi \in \mathcal{U}$. Then,  a.s., for all $0 \leq t \leq \tau \leq T$, we have
\begin{small}
\begin{align*}
v_\tau(X_\tau^{t,x,\pi})&=v_t(x)+\int_t^\tau \mathfrak{d}_{s}v_{s}(X_s^{t,x,\pi})+ D v_s(X_s^{t,x,\pi})b_s(X_s^{t,x,\pi},\pi_s) + \operatorname{tr} \left( a_s(X_s^{t,x,\pi},\pi_s) D^2 v_s(X_s^{t,x,\pi}) \right)\\
&\quad + \gamma_s(X_s^{t,x,\pi},\pi_s) \cdot D(\mathfrak{d}_{\omega}v)_{s}(X_s^{t,x,\pi},\pi_s) ds+\int_t^\tau Dv_s(X_s^{t,x,\pi}) \sigma_s(X_s^{t,x,\pi},\pi_s) dB_s\\
&\quad+ \int_t^\tau\left[ \mathfrak{d}_{\omega}v_{s
}(X_s^{t,x,\pi}) + Dv_s(X_s^{t,x,\pi})\gamma_s(X_s^{t,x,\pi},\pi_s) \right]dW_s.
\end{align*}
\end{small}
\end{lemma}
In \cref{subsection_Viscosity Solutions to the Rough Stochastic Control Problems} it is necessary for us to work with versions of functionals as in \cref{def:smooth-semimartingales} with higher sample path regularity. The following claim shows that in the limit we may always pick such a version on a compact domain. 
\begin{lemma}
\label{appendix_smooth_time}
    Let $\mathbb{B}$ be a separable Banach space. Then for any $p \in [1, \infty)$ and any $H \in \mathcal{L}^p(\mathbb{B})$, there is a sequence of processes $(H^n)_{n \in \mathbb{N}}\subset \mathcal{L}^p(\mathbb{B})$ with a.s. continuous sample paths, such that $\Vert H-H^n \Vert_{\mathcal{L}^p}\to 0$. In particular $\mathcal{S}^p(\mathbb{B})$ is dense in $\mathcal{L}^{p}(\mathbb{B})$. 
\end{lemma}
\begin{proof}
    First of all denote by $\mathbf{S}(\mathbb{B})$ the set of simple $\mathbb{B}$-valued processes. Recall that for any $H \in \mathcal{L}^p(\mathbb{B})$ it holds $H: [0,T] \times \Omega \to \mathbb{B}$ is $\mathfrak{P}/\mathfrak{B}(\mathbb{B})$-measurable. Denoting by $X=([0,T] \times \Omega, \mathfrak{P}, \operatorname{Leb} \otimes \mathbb{P})$ the underlying measure space we see that $L^p_{X}(\mathbb{B})=\mathcal{L}^p(\mathbb{B})$ in the sense of Bochner-integration (since $\mathbb{B}$ is separable, all elements of $\mathcal{L}^{p}(\mathbb{B})$ are Bochner-measurable). By the construction of the Bochner integral this implies the existence of a sequence of simple processes of the form 
   $H^{n}= \sum_{j=0}^{N(n)} b^n_{j} \mathbbm{1}_{A^n_{j}}$
    approximating $H$ in $L^P(\mathbb{B})$ with $(b_{j}^n)_{j=0}^{N(n)}\subset \mathbb{B}$ and disjoint $(A_{j}^n)_{j=0}^{N(n)}\in \mathfrak{P}$ for any $n \in \mathbb{N}$. By a Monotone Class argument these sets can be taken to be of the form $A_{j}^{n}= B_{j}^{n} \times (t_{j}^{n} , t_{j=1}^{n}]$ for some $B_{j}^{n} \in \mathfrak{F}_{t^n_{j}}$ and $(t_{j}^n)_{j=0}^{N(n)}$ being some disjoint  deterministic partition of $[0,T]$. Therefore $\mathbf{S}(\mathbb{B})\subset \mathcal{L}^P(\mathbb{B})$ is dense w.r.t. $\Vert \cdot \Vert_{\mathcal{L}^p(\mathbb{B})}$. Thus we may rewrite the sequence $(H^{n})_{n \in \mathbb{N}}$ as
    \begin{equation*}
        H^n(t, \omega)= \sum_{j=0}^{N(n)} b^n_{j} \mathbbm{1}_{t \in (t_{j}^{n}, t_{j+1}^{n}]} \mathbbm{1}_{\omega \in B_{j}^{n}}=: \sum_{j=0}^{N(n)} H^{n}_{j}\mathbbm{1}_{t \in (t_{j}^{n}, t_{j+1}^{n}]} 
    \end{equation*}
    with $H^{n}_{j}$ being $\mathfrak{F}_{t_{j}}$-measurable $\mathbb{B}$-valued random variables. By an application of Urysohn's Lemma and using that the Lebesgue-measure is inner regular on $([0,T] , \mathfrak{B}_{[0,T]})$ we can now approximate the indicator functions by continuous functions in Lebesgue-measure, which yields the claim.
\end{proof}
Additionally we can also restrict to dense subspaces of the Banach space $\mathbb{B}$. 
\begin{lemma}
\label{appendix_smooth_time_space}
    Let $\mathbb{B}$ be a separable Banach space and $\mathbb{A} \subset \mathbb{B}$ be a dense subspace. Then for any $H \in \mathcal{L}^p(\mathbb{B})$ there is a sequence of $(H^n)_{n \in \mathbb{N}}\subset \mathcal{L}^p(\mathbb{A})$ such that $\Vert H^n - H \Vert_{\mathcal{L}^p(\mathbb{B})}\to 0$. This sequence can also be taken to have continuous sample paths almost surely. In particular $\mathcal{S}^p(\mathbb{A})$ is dense in $\mathcal{L}^{p}(\mathbb{B})$.
\end{lemma}
\begin{proof}
    The proof works essentially the same as the proof of \cref{appendix_smooth_time} only that we approximate the $\mathbb{B}$-valued coefficients of the simple processes by $\mathbb{A}$-valued coefficients and then proceed as in the above proof. 
\end{proof}
\begin{coroll}
\label{approx_coroll_smooth_semimart}
    For any $v \in \mathcal{C}^2_{\mathfrak{F}}$ and $K \subset \mathbb{R}^{d}$ open bounded set there are sequences \\ $(v^{n, j})_{n},(\mathfrak{d}_{t} v^n)_{n},(\mathfrak{d}_{w} v^{n,i})_{n} \subset \mathcal{S}^{2}(\Lip(\mathbb{R}^{d}))$ for $j=0,1,2$ and $i=0,1$ such that 
    \begin{equation}
    \label{eq:appendix_convergence}
        \Vert D^{j}v - v^{n,j} \Vert_{\mathcal{L}^{2}(C(K))} , \Vert \mathfrak{d}_{t} v - \mathfrak{d}_{t} v^n \Vert_{\mathcal{L}^{2}(C(K))}, \Vert D^{i}(\mathfrak{d}_{w} v) - \mathfrak{d}_{w} v^{i,n} \Vert_{\mathcal{L}^{2}(C(K))} \stackrel{ n \to \infty}{\longrightarrow }0
    \end{equation}
     and for any $n \in \mathbb{N}$ it holds
    \begin{equation}
    \label{eq:appendix_bounds_Tietze}
    \begin{aligned}
        &\Vert v^{n,j} \Vert_{\mathcal{L}^{2}(C_{b}(\mathbb{R}^{d}))}\leq  \Vert v \Vert_{\mathcal{L}^{2}(C^2_{b}(\mathbb{R}^{d}))}+ \max_{j=0,1,2} \Vert D^{j}v - v^{n,j}\Vert_{\mathcal{L}^{2}(C(K))},  \\
        &\Vert \mathfrak{d}_{t} v^n \Vert_{\mathcal{L}^{2}(C_{b}(\mathbb{R}^{d}))}\leq \Vert \mathfrak{d}_{t} v \Vert_{\mathcal{L}^{2}(C_{b}(\mathbb{R}^{d}))}+ \Vert \mathfrak{d}_{t} v - \mathfrak{d}_{t} v^n \Vert_{\mathcal{L}^{2}(C(K))},\\
        &\Vert \mathfrak{d}_{w} v^{n,i} \Vert_{\mathcal{L}^{2}(C_{b}(\mathbb{R}^{d}))}\leq \Vert \mathfrak{d}_{w} v \Vert_{\mathcal{L}^{2}(C^{1}_{b}(\mathbb{R}^{d}))} +  \Vert D^{i}(\mathfrak{d}_{w} v) - \mathfrak{d}_{w} v^{n, i
        }\Vert_{\mathcal{L}^{2}(C(K))}.
    \end{aligned}
    \end{equation}
\end{coroll}
\begin{proof}
The existence of $(v^n,\mathfrak{d}_{t} v^n, \mathfrak{d}_{w} v^n) \in \mathcal{S}^{2}(C^3(K)) \times \mathcal{S}^2(C^1(K)) \times \mathcal{S}^2(C^2(K))$ satisfying \eqref{eq:appendix_convergence} with $v^{n, j}\coloneqq D^{j}v^{n}$ and $\mathfrak{d}_{\omega}v^{n, i}\coloneqq D^{i} (\mathfrak{d}_{\omega} v^{n})$ follows directly by applying \cref{appendix_smooth_time_space} to the triplet $(v, \mathfrak{d}_{t} v, \mathfrak{d}_{w} v) \in \mathcal{L}^2(C^2(K)) \times \mathcal{L}^2\left(C(K)\right) \times \mathcal{L}^2\left(C^1\left(K\right)\right) $. We then extend $g \in \{v^{n,j}_{t}(\omega, \cdot),\mathfrak{d}_{t} v^n_{t}(\omega, \cdot), \mathfrak{d}_{w} v^{n,i}_{t}(\omega, \cdot) \}$ onto $\mathbb{R}^{d}$ by setting $\tilde{g}(x)\coloneqq \sup_{y \in K}(g(y)-  |x-y|)$  and $g(x)\coloneqq ( -\Vert g \Vert_{C(K)})\vee (\tilde{g}(x) \wedge \Vert g \Vert_{C(K)}) $ for any $x \in \mathbb{R}^{d}$. This also yields the bounds \eqref{eq:appendix_bounds_Tietze}
\end{proof}
\begin{small}
\printbibliography
\end{small} 
\end{document}